\documentclass{tran-l}
\usepackage{fourier}
\copyrightinfo{}{}

\newif\ifnmd
\InputIfFileExists{user-is-nathan-dunfield}{\nmdtrue}{\nmdfalse}

\usepackage[colorlinks=true, linkcolor=blue, citecolor=blue, urlcolor=blue]{hyperref}
\usepackage{graphicx}
\usepackage[all,import,rotate]{xy}
\usepackage[margin=10pt,font=small,labelfont=bf]{caption}
\usepackage{ragged2e}

\def\RCS$#1: #2 ${\expandafter\def\csname RCS#1\endcsname{#2}}
\RCS$Revision: 136 $
\RCS$Date: 2011-07-03 08:06:50 -0500 (Sun, 03 Jul 2011) $

\def\MR#1{\href{http://www.ams.org/mathscinet-getitem?mr=#1}{MR#1}}
\def\checkMR MR#1#2#3 #4\relax%
  {\ifx#1M%
     \ifx#2R\MR{#3}\else\MR{#1#2#3}\fi
   \else
     \MR{#1#2#3}%
   \fi}

\theoremstyle{plain}
\swapnumbers
\newtheorem{theorem}[subsection]{Theorem}

\newtheorem{lemma}[subsection]{Lemma}

\newtheorem{proposition}[subsection]{Proposition}
\newtheorem{claim}[subsection]{Claim}

\newtheorem{mainonebackwards}{Theorem}

\newtheorem{maintwobackwards}{Theorem}

\theoremstyle{definition}
\newtheorem{definition}[subsection]{Definition}

\theoremstyle{remark}
\newtheorem{remark}[subsection]{Remark}
\newtheorem{example}[subsection]{Example}

\makeatletter
  \let\c@figure=\c@subsection
  
  \let\c@table=\c@subsection
  
  \let\c@equation=\c@subsection
  
\makeatother

\newcommand{\hyp}{\nobreakdash-\hspace{0pt}}
\newcommand{\3}[1]{3\hyp}

\newenvironment{xyoverpic}[3]
{%
\begin{xy}
\xyimport#1{\includegraphics[#2]{#3}}
}{\end{xy}}

\newenvironment{cxyoverpic}[3]
{%
\begin{center}
\centering\leavevmode\small
\begin{xyoverpic}{#1}{#2}{#3}
}{\end{xyoverpic}
\end{center}}

\def\BZ{\mathbb Z}
\def\BP{\mathbb P}
\def\BQ{\mathbb Q}
\def\BR{\mathbb R}
\def\T{\mathcal T}
\newcommand{\Q}{{\mathbb Q}}
\newcommand{\R}{{\mathbb R}}
\newcommand{\C}{{\mathbb C}}
\newcommand{\Z}{{\mathbb Z}}
\newcommand{\N}{{\mathbb N}}
\renewcommand{\H}{{\mathbb H}}

\newcommand{\maps}{\colon\thinspace}
\newcommand{\tr}{ {\mathrm{tr}} }

\newcommand{\PSL}[2]{\mathrm{PSL}_{#1} #2}

\newcommand{\mtext}[1]{\quad\mbox{#1}\quad}
\DeclareMathOperator{\rank}{rank}

\newcommand{\pair}[1]{\left\langle #1 \right\rangle}
\newcommand{\setdef}[2]{{  \left\{  {#1}  \ \left| \   {#2} \right. \right\} }}

\newcommand{\kernoverline}[3]{{\mkern #1mu\overline{\mkern-#1mu #2\mkern-#3mu}\mkern#3mu}}

\newcommand{\Coo}{\C^{* \bullet}}
\newcommand{\Co}{\C^{\bullet}}
\newcommand{\Vbar}{\kernoverline{2}{V}{1}}
\newcommand{\Wbar}{\kernoverline{2}{W}{1}}
\newcommand{\Ubar}{\kernoverline{2}{U}{1}}
\newcommand{\dd}{d}

\newcommand{\bs}[1]{\mathrm{bs}(#1)}
\newcommand{\Dtil}{\widetilde{D}}
\newcommand{\Ctil}{\widetilde{C}}
\newcommand{\Wtil}{\widetilde{W}}
\newcommand{\Mtil}{\widetilde{M}}
\newcommand{\dtil}{\widetilde{d}}

\newcommand{\Ntil}{\widetilde{N}}
\newcommand{\Ndot}{N^\bullet}
\newcommand{\Ntildot}{\widetilde{N}^\bullet}
\newcommand{\Nbar}{\kernoverline{4}{N}{2}}
\newcommand{\Hbar}{\kernoverline{2}{\H}{2}}
\newcommand{\Xbar}{\kernoverline{3}{X}{1}}
\newcommand{\Tbar}{\kernoverline{2}{\T}{1}}
\newcommand{\Tbarprime}{\kernoverline{2}{\T}{1} \mkern 1mu '}
\renewcommand{\v}{\mathfrak{v}}
\newcommand{\n}{\mathfrak{n}}
\newcommand{\Nv}{N_{\v}}
\newcommand{\Bv}{B_{\v}}
\newcommand{\Deltav}{\Delta_\v^3}
\newcommand{\TN}{\mathfrak{T}_N}
\newcommand{\TS}{\mathfrak{T}_S}

\begin{document}

\title[Incompressibility criteria for spun-normal surfaces]{Incompressibility criteria for \\ spun-normal surfaces}

\author{Nathan M. Dunfield}
\address{ Dept. of Math., MC-382 \\
          University of Illinois \\
          1409 W. Green Street \\
          Urbana, IL 61801, USA
}
\email{nathan@dunfield.info}
\urladdr{http://dunfield.info}

\author{Stavros Garoufalidis}
\address{School of Mathematics \\
         Georgia Institute of Technology \\
         Atlanta, GA 30332-0160, USA}

\email{stavros@math.gatech.edu}
\urladdr{http://www.math.gatech.edu/~stavros}



\begin{abstract}
  We give a simple sufficient condition for a spun-normal surface in
  an ideal triangulation to be incompressible, namely that it is a
  vertex surface with non-empty boundary which has a quadrilateral in
  each tetrahedron.  While this condition is far from being
  necessary, it is powerful enough to give two new results: the
  existence of alternating knots with non-integer boundary slopes, and a
  proof of the Slope Conjecture for a large class of 2-fusion knots.  

  While the condition and conclusion are purely topological, the proof
  uses the Culler-Shalen theory of essential surfaces arising from
  ideal points of the character variety, as reinterpreted by Thurston
  and Yoshida.  The criterion itself comes from the work of Kabaya,
  which we place into the language of normal surface theory.  This
  allows the criterion to be easily applied, and gives the framework
  for proving that the surface is incompressible.

  We also explore which spun-normal surfaces arise from ideal points
  of the deformation variety.  In particular, we give an example where
  no vertex or fundamental surface arises in this way.
\end{abstract}

\maketitle

%\vspace{-0.5cm}
\setcounter{tocdepth}{1}
\tableofcontents
%\vspace{-0.5cm}

\section{Introduction}

\label{sec.intro}
 
Let $M$ be a compact oriented 3-manifold whose boundary is a torus.
A properly embedded surface $S$ in $M$ is called \emph{essential} if
it is incompressible, boundary\hyp incompressible, and not
boundary-parallel.  If $S$ has boundary, this consists of
pairwise-isotopic essential simple closed curves on the torus $\partial M$;
the unoriented isotopy class of these curves is the \emph{boundary
  slope} of $S$.  Such slopes can be parameterized by the
corresponding primitive homology class in $H_1(\partial M; \Z)/(\pm1)$; if a
basis of $H_1(\partial M; \Z)$ is fixed, slopes can also be recorded as
elements of $\BQ \cup\{\infty \}$. 

Our focus here is on the set $\bs M$ of all boundary slopes of
essential surfaces in $M$, which is finite by a fundamental result of
Hatcher~\cite{H}.  This is an important invariant of $M$, for instance
playing a key role in the study of exceptional Dehn filling.  Building
on Haken's fundamental contributions \cite{Hk}, Jaco and Sedgwick
\cite{JS} used normal surface theory to give a general algorithm for
computing $\bs M$.  As with most normal-surface algorithms, this
method seems impractical even for modest-sized examples (however, some
important progress has been made on this by \cite{BRT}).  For certain
special cases, such as exteriors of Montesinos knots, fast algorithms
do exist \cite{HT, HO, DunfieldMontesinos}, and additionally 
character-variety techniques
can sometimes be used to find boundary slopes \cite{CCGLS, Culler}.
However, there remain quite small examples where $\bs M$ is unknown,
e.g.~for the exteriors of certain 9-crossing knots in $S^3$.

Here, we introduce a simple sufficient condition that ensures that a
normal surface is essential.  While our condition is far from being
necessary, it is powerful enough to give two new results: the
existence of alternating knots with non-integer boundary slopes, and
a proof of the Slope Conjecture for all 2-fusion knots.  Along with
\cite{BRT}, these are the first results that come via applying
directly normal surface algorithms, which have been much studied for
their inherent interest in the past 50 years.
 
We work in the context of an ideal triangulation $\T$ of $M$ and
Thurston's corresponding theory of spun-normal surfaces (throughout,
see Section \ref{sec-spun-normal} for definitions).  In normal surface
theory, \emph{vertex surfaces} corresponding to the vertices of the
projectivized space of normal surfaces play a key role.  Our basic
result is

\begin{theorem}\label{thm-one-cusp}
  Suppose $S$ is a vertex spun-normal surface in $\T$ with non-trivial
  boundary.  If $S$ has a quadrilateral in every tetrahedron of $\T$,
  then $S$ is essential.   
\end{theorem}

While this statement is purely topological, the proof uses the
Culler-Shalen theory of essential surfaces arising from ideal points
of the character variety \cite{CS, CGLS}, as reinterpreted by Thurston
\cite{T} and Yoshida \cite{Y} in the context of the deformation
variety defined by the hyperbolic gluing equations for $\T$.
Theorem~\ref{thm-one-cusp} is a strengthening of a result of Kabaya
\cite{Ka}, who shows that, with the same hypotheses, that the boundary
slope of $S$ is in $\bs{M}$.  Our contribution to
Theorem~\ref{thm-one-cusp} is restating Kabaya's work in the language
of normal surface theory, allowing it to be easily applied, and
showing that $S$ is itself incompressible.

\subsection{Alternating knots}

Our application of Theorem~\ref{thm-one-cusp} concerns the
boundary slopes of (the exteriors of) alternating knots in $S^3$.  In
the natural meridian-longitude basis for $H_1(\partial M)$, Hatcher
and Oertel \cite{HO} showed that the boundary slopes of alternating
Montesinos knots were always even integers, generalizing what Hatcher
and Thurston had found for 2-bridge knots \cite{HT}.  Hatcher and
Oertel asked whether this was true for \emph{all} alternating knots.
We use Theorem~\ref{thm-one-cusp} to settle this 20 year-old question:
\begin{theorem}\label{thm-alt-nonint}
  There are alternating knots with nonintegral boundary slopes.  In
  particular, the knot $10_{79}$ has boundary slopes $10/3$ and
  $-10/3$.
\end{theorem}
\noindent
Many additional such examples are listed in Table~\ref{table-alt}.

\subsection{Dehn filling}

The technique of Kabaya that underlies Theorem~\ref{thm-one-cusp} can
be generalized to manifolds that arise from Dehn filling all but one
boundary component of a more complicated manifold.  Specifically, in
the language of Section~\ref{sec-filling} we show: 
\begin{theorem}\label{thm-filling}
  Let $W$ be a compact oriented \3-manifold whose boundary consists of
  tori $T_0, T_1, \ldots , T_n$.  Let $S$ be a spun-normal surface in an
  ideal triangulation $\T$ of $W$, with nonempty boundary slope $\gamma_k$
  on each $T_k$. Suppose that $S$ has a quadrilateral in every
  tetrahedra of $\T$, and is a vertex surface for the relative normal
  surface space corresponding to $(\, \cdot \,, \gamma_1,\ldots,\gamma_n)$.  Then $\gamma_0$ is
  a boundary slope of $W(\cdot, \gamma_1,\ldots, \gamma_n)$.
\end{theorem}
This broadens the applicability of Kabaya's approach since any given
$M$ arises in infinitely many ways by Dehn filling, and thus a fixed
surface $S \subset M$ has many chances where Theorem~\ref{thm-filling}
might apply.

\subsection{The Slope Conjecture for 2-fusion knots}

Our application of Theorem~\ref{thm-filling} involves constructing a
boundary slope for every knot of a certain 2-parameter family.  We use
this to prove the Slope Conjecture of \cite{Ga1} in the case of
2-fusion knots.  This conjecture relates the degree of the Jones
polynomial of a knot and its parallels to boundary slopes of essential
surfaces in the knot complement.  To state our result, consider the
3-component link $L$ from Figure \ref{fig-2fusion-link}.
\begin{figure}
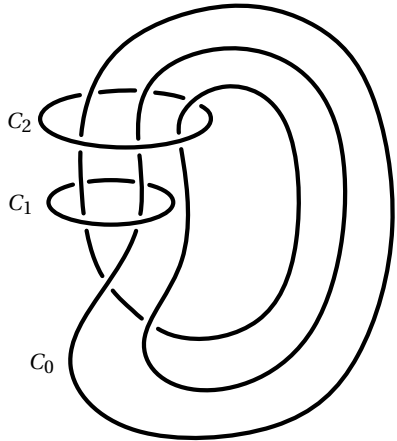

  \begin{cxyoverpic}{(169,206)}{scale=0.80}{images/2fusion-link}
    ,(1,150)*+!R{C_2}
    ,(5,111)*++!R{C_1}
    ,(15,36)*++!R{C_0}
  \end{cxyoverpic}
  \caption{The link $L$.}
  \label{fig-2fusion-link}
\end{figure}
For a pair of integers $(m_1,m_2) \in \BZ^2$, let $L(m_1,m_2)$ denote
the knot obtained by $(-1/m_1,-1/m_2)$ filling on the cusps $C_1$ and
$C_2$ of $L$, leaving the cusp $C_0$ unfilled. The 2-parameter family
of knots $L(m_1,m_2)$, together with the double-twist knots coming
from filling 2 cusps of the Borromean link, is the set of all knots of
fusion number at most $2$; see \cite{Ga2}.  The family $L(m_1,m_2)$
has some well-known members: $L(2,1)=L(-1,2)$ is the $(-2,3,7)$
pretzel knot, $L(-2,1)$ is the $5_2$ knot, $L(-1,3)$ is $k4_3$ which
was the focus of \cite{GL}, and $L(m_1,1)$ is the $(-2,3,3m_1+3)$
pretzel knot.
\begin{figure}[tpb]
\begin{center}
\includegraphics[height=1.25in]{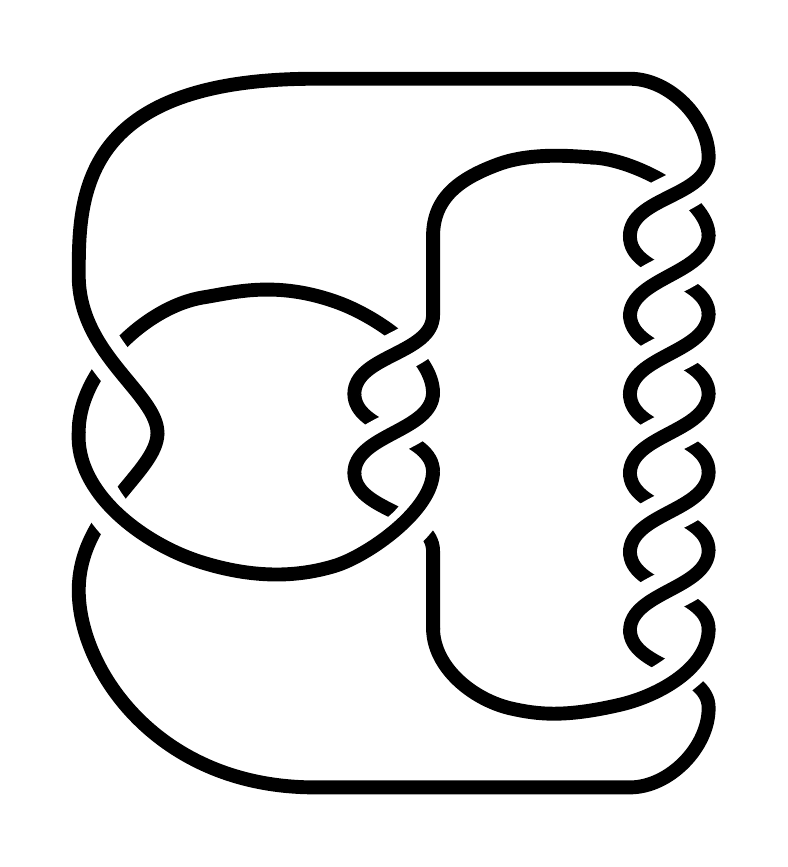} \hspace{1cm} %
\includegraphics[height=1.1in]{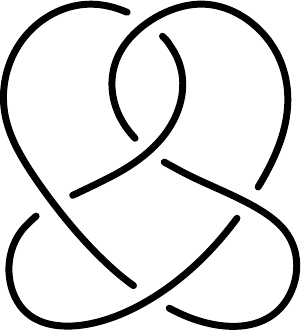} \hspace{1cm} %
\includegraphics[height=1.1in]{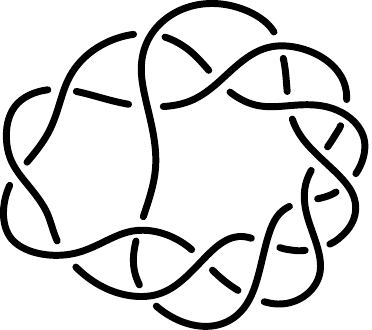}
\end{center}
\caption{The $(-2,3,7)$ pretzel, the $5_2$ knot, and the $k4_3$ 
knot.}\label{fig.2knots}
\end{figure}

Together with the results of \cite{Ga2}, the following confirms the
Slope Conjecture for 2-fusion knots in one of three major cases:
\begin{theorem}
\label{thm.Lm1m2}
For $m_1>1, m_2 >0$, one boundary slope of $L(m_1,m_2)$ is:
\begin{equation*}
3 (1 + m_1) + 9 m_2 + \frac{(m_1 -1)^2}{m_1 + m_2 -1}
\end{equation*}
\end{theorem}
It is mysterious how the Jones polynomial selects one (out of the
many) boundary slopes of a knot, and it was fortunate that this slope
happens to be one of the few accessible by the special method of
Theorem~\ref{thm-filling} for the family $L(m_1,m_2)$ of 2-fusion
knots.  Indeed, we tried without success to apply our same method to
confirm the Slope Conjecture for the rest of the 2-fusion knots.  Note
also that the results of \cite{FuterKalfagianniPurcell2011} do not
imply Theorem~\ref{thm.Lm1m2} as the former only produce integer
boundary slopes.

\subsection{Technical results} 

In addition to Theorems~\ref{thm-one-cusp} and \ref{thm-filling}, we
make progress on the question of which spun-normal surfaces in an
ideal triangulation $\T$ arise from an ideal point of the deformation
variety $D(\T)$ (see Section~\ref{sec-def-var} for more on the
latter).  In particular, given an ideal triangulation $\T$ of a
manifold $M$ with one torus boundary component, the goal is to
determine all the boundary slopes that arise from ideal points of
$D(\T)$.  Of course, one can find all such detected slopes by
computing the $A$-polynomial, but this is often a very difficult
computation, involving projecting an algebraic variety
(i.e.~eliminating variables).

For a fixed surface $S$, we give a relatively easy-to-check
algebro-geometric condition (Lemma~\ref{lem-gen-char}) which is both
necessary and sufficient for $S$ to come from an ideal point.  However,
there are often only finitely many ideal points but infinitely many
spun-normal surfaces, and so Lemma~\ref{lem-gen-char} does not
completely solve this problem.  A natural hope is that the surfaces
associated to ideal points would be vertex or fundamental surfaces,
but we give a simple example in Section~\ref{sec-funny-knot} where
this is not the case.

\subsection{Outline of contents}

In Sections~\ref{sec-spun-normal} and \ref{sec-def-var} we review the
basics of spun-normal surfaces and deformation varieties.  Then in
Section~\ref{sec-ideal-gluing}, we study a class of algebraic
varieties which includes these deformation varieties.  We place
Kabaya's motivating result into that context
(Proposition~\ref{prop-Kabaya-key}) and also give a necessary and
sufficient condition for there to be an ideal point with certain data
(Lemma~\ref{lem-gen-char}).  Section~\ref{sec-proof-one-cusp} is
devoted to the proof of Theorem~\ref{thm-one-cusp}, and then
Section~\ref{sec-alt} applies this result to give non-integral
boundary slopes for alternating knots.  Likewise,
Section~\ref{sec-filling} proves Theorem~\ref{thm-filling} and
Section~\ref{sec.thmLm1m2} applies it to the Slope Conjecture for
  2-fusion knots.  Finally, Section~\ref{sec-which-come-from-ideal}
  explores the effectiveness and limitations of the methods studied
  here.

\subsection{Acknowledgments}     

The authors were partially supported by the U.S. National Science
Foundation, via grants DMS-0707136 and DMS-0805078 respectively.  We
thank Marc Culler, Tom Nevins, Fernando Rodriguez Villegas, Hal
Schenck, Saul Schleimer, Eric Sedgwick, Henry Segerman, Bernd
Sturmfels, Stephan Tillmann, and Josephine Yu for helpful
conversations, and the organizers of the Jacofest conference for their
superb hospitality.   We also thank the referee for their very detailed
and helpful comments on the original version of this paper.  

\section{Spun-normal surfaces}\label{sec-spun-normal}

In this section, we sketch Thurston's theory of spun-normal surfaces
in ideal triangulations.  We follow Tillmann's exposition \cite{Till1}
which contains all the omitted details (see also \cite{Kang, KR}).
Let $M$ be a compact oriented 3-manifold whose boundary is a nonempty
union of tori.  An ideal triangulation $\T$ of $M$ is a $\Delta$-complex
(in the language of \cite{HatcherBook}) made by identifying faces of
\3-simplices in pairs so that $\T \setminus(\mbox{vertices})$ is homeomorphic
to $\mbox{int}(M)$.  Thus $\T$ is homeomorphic to $M$ with each
component of $\partial M$ collapsed to a point.

A spun-normal surface $S$ in $\T$ is one which intersects each
tetrahedron in finitely many quads and infinitely many triangles
marching out toward each vertex (see Figure~\ref{fig-tet-with-spun}(a)).
While there are infinitely many pieces, in fact $S$ is typically the
interior of a properly embedded compact surface in $M$ whose boundary has been
``spun'' infinitely many times around each component of $\partial M$.  (The
other possibility for $S$ near a vertex is that it consists of
infinitely many disjoint boundary-parallel
tori.)
\begin{figure}
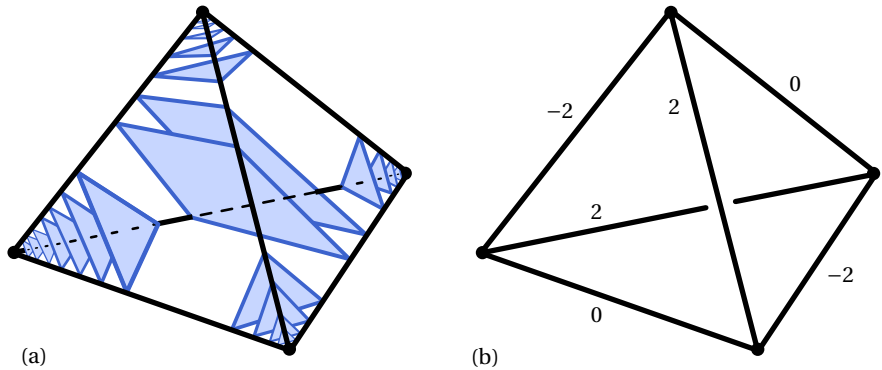

  \begin{cxyoverpic}{(329,133)}{scale=1.0}{images/spun-normal-tet}
    ,(222,49)*+!D{2}
    ,(259,94)*++!R{2}
    ,(216,87)*+!DR{-2}
    ,(306,36)*+!UL{-2}
    ,(227,21)*+!UR{0}
    ,(292,97)*+!DL{0}
    ,(10,6)*+!U{\mbox{(a)}}
    ,(180,6)*+!U{\mbox{(b)}}
 \end{cxyoverpic}
  \caption{At left is the intersection of a spun-normal surface with a
    single tetrahedron, with infinitely many triangles in each corner.  At
    right are the edge shifts of the hexagon regions as defined in
    Figure~\ref{fig-shifts}.}
  \label{fig-tet-with-spun}
\end{figure}
 \begin{figure}
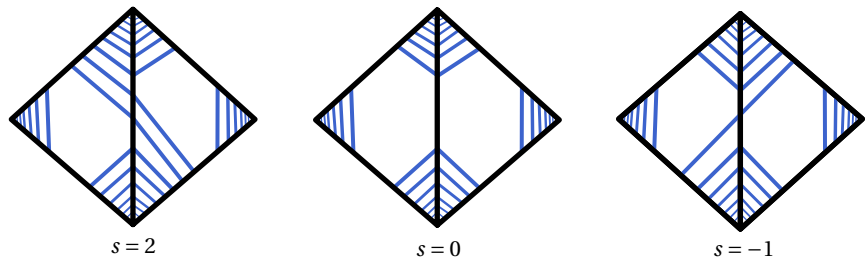

  \begin{cxyoverpic}{(323,86)}{scale=1.0}{images/shifts}
    ,(47,3)*++!U{s = 2}
    ,(162,2)*++!U{s = 0}
    ,(277,2)*++!U{s = -1}
  \end{cxyoverpic}
  \vspace{-10pt}
  \caption{The shift parameter $s$ of an edge describes the relative
    positions of the adjacent hexagons, as viewed from outside the
    simplex $\Delta$.  Here $\Delta$ is oriented by the orientation of
    $M$, and the induced orientation of $\partial\Delta$ distinguishes
    between $s > 0$ and $s < 0$.  As each picture is invariant under
    rotation by $\pi$, the shift does \emph{not} depend on an
    orientation of the edges themselves.  The convention here agrees
    with \cite{Till1}.}
  \label{fig-shifts}
\end{figure}
Notice from Figure~\ref{fig-tet-with-spun}(a) that on any face of a
tetrahedron, there is exactly one hexagon region and infinitely many
four-sided regions.  Thus to specify a spun-normal surface $S$, we
need only record the number and type of quads in each tetrahedron of
$\T$, since the need to glue hexagons to hexagons uniquely specifies
how the local pictures of $S$ must be glued together across adjoining
tetrahedra.  As there are three kinds of quads, if $\T$ has $n$
tetrahedra then $S$ is uniquely specified by a vector in $\Z_+^{3n}$
called its $Q$-coordinates.  This vector satisfies certain linear
equations which we now describe as they will explain how ideal points
of the deformation variety give rise to such surfaces.

For an edge of a tetrahedron, let $s$ be the amount the adjacent
hexagons are shifted relative to each other; the orientation
convention is given in Figure~\ref{fig-shifts}, and
Figure~\ref{fig-tet-with-spun}(b) shows the resulting shifts on all
edges of a tetrahedron.
It is not so hard to see that $v \in \Z_+^{3n}$ corresponds to a spun-normal
surface if and only if
\begin{enumerate}
  \item \label{cond-add}
    There is at most one non-zero quad weight in any given tetrahedron.

  \item \label{cond-Q} 
    As we go once around an edge, the positions of the hexagons match up.
    That is, the sum of the shifts $s$ must be $0$.
\end{enumerate}
The shifts are linear functions of the entries of $v$ (see
Figure~\ref{fig-tet-with-spun}), and so the conditions in
(\ref{cond-Q}) form a linear system of equations called the
\emph{Q-matching equations}.

As their $Q$-coordinates satisfy various linear equalities and
inequalities, spun-normal surfaces fit into the following geometric
picture.  Let $C(\T)$ be the intersection of $\R_+^{3n}$ with the
subspace of solutions to the $Q$-matching equations.  Thus $C(\T)$ is
a finite-sided convex cone.  If we impose condition (\ref{cond-add})
as well, we get a set $F(\T)$ which is a finite union of convex cones
whose integral points are precisely the $Q$-coordinates of spun-normal
surfaces.  Within each convex cone of $F(\T)$, vector addition
of $Q$-coordinates corresponds to a natural geometric sum operation on
the associated spun-normal surfaces.

It is natural to projectivize $F(\T)$ by intersecting it with the
affine subspace where the coordinates sum to 1.  The resulting set
$\mathit{PF}(\T)$ is a finite union of compact polytopes.  Since all the
defining equations had integral coefficients, the vertices of these
polytopes lie in $\Q^{3n}_+$.  For such a vertex $v$, consider the
smallest rational multiple of $v$ which lies in $\Z^{3n}_+$; that vector
gives a spun-normal surface, called a \emph{vertex surface}.  Vertex
surfaces play a key role in normal surface theory generally and here
in particular.

One major difference between spun-normal surface theory and the
ordinary kind for non-ideal triangulations is that normalizing a given
surface is much more subtle.  This is because of the infinitely many
intersections of a spun surface with the 1-skeleton of $\T$.  However,
building on ideas of Thurston, Walsh has shown that essential surfaces
which are not fibers or semi-fibers can be spun-normalized, using
characteristic submanifold theory \cite{Walsh}.  Despite this, some
key algorithmic questions remain unanswered for spun-normal surfaces.
For instance, when $M$ has one boundary component, do all the strict
boundary slopes arise from \emph{vertex} spun-normal surfaces which
are also essential?  For an ordinary triangulation of $M$ (which will
typically have more tetrahedra than an ideal one), the answer is yes
\cite[Theorem~5.3]{JS}.

\subsection{Ends of spun-normal surfaces} \label{sec-ends-spun}

We now describe how a spun-normal surface gives rise to a properly
embedded surface in $M$, closely following Sections 1.9-1.12 of
\cite{Till1}.  For notational simplicity, we assume that $M$ has a
single boundary component.  Let $\v$ be the vertex of $\T$, and
consider a small neighborhood $\Nv$ of $\v$ bounded by a normal torus
$\Bv$ consisting of one normal triangle in each corner of every
tetrahedron in $\T$.  We can assume that $\Bv$ and $S$ are in general
position and that $\Nv$ meets only normal triangles of $S$.
 
We put a canonical orientation on the curves of $S \cap \Bv$ as
follows.  First, triangulate $\Nv$ by taking the cone to $\v$ of the
triangulation of $\Bv$.  If $\n$ is a normal triangle of $S$ meeting
$\Nv$, its interior meets exactly one tetrahedron $\Deltav$ in $\Nv$.
\begin{figure}[tb]
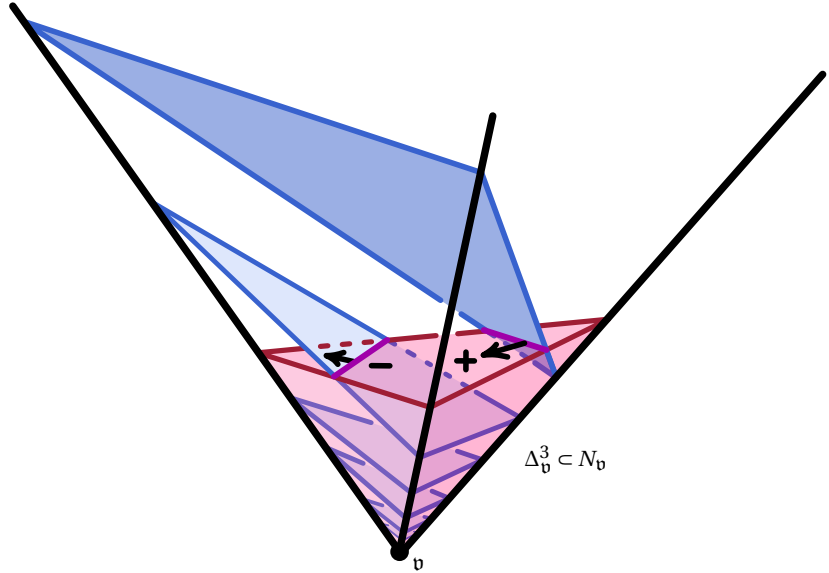

  \begin{cxyoverpic}{(233,154)}{scale=1.4}{images/spun-ends}
    ,(106,2)*++!L{\v}
   ,(138,30)*+!L{\Deltav \subset \Nv}
  \end{cxyoverpic}
 \caption{Orienting $S \cap \Bv$.  Notice that $S$ meets the triangle
    of $\Bv$ in at most two of the three possible types of normal
    arcs. }\label{fig-bdry-orient}
\end{figure}  
We orient $\n$ by assigning $+1$ to the component of $\Deltav
\setminus \n$ which contains $\v$.  This induces a consistent
transverse orientation for each component of $ S \cap \Bv$ as shown in 
Figure~\ref{fig-bdry-orient}.

By Lemma~1.31 of \cite{Till1}, we can also do a normal isotopy of $S$
so that all the components of $S \cap \Bv$ are nonseparating in the
torus $\Bv$.  If the components of $S \cap \Bv$ don't all have the
same orientation, apply the proof of Lemma~1.31 of \cite{Till1} to an
annulus between two adjacent components with opposite orientations to
reduce the size of $S \cap \Bv$.  Thus we can assume that all
components of $S \cap \Bv$ have the same orientation.  It then follows
from Lemma~1.35 of \cite{Till1} that $S \cap \Nv$ consists of parallel
half-open annuli spiraling out toward $\v$.

We now identify $\T \setminus \mathrm{int}({\Nv})$ with $M$.  Then $S'
= S \cap M$ is a properly embedded surface in $M$.  Since we
understand $S \cap \Nv$, it's easy to see that the isotopy type of
$S'$ is independent of the choice of such $\Nv$.  (Here isotopies of $S'$
are allowed to move $\partial S'$ within $\partial M$; the isotopy
class of $S'$ with $\partial S'$ fixed typically does depend on the
choice of $\Nv$.)  Thus, it makes sense to talk about the number of
boundary components of $S$ and their slope.

\section{Deformation varieties} \label{sec-def-var}

As in Section~\ref{sec-spun-normal}, let $\T$ be an ideal
triangulation of a compact oriented 3-manifold $M$ with boundary a
union of tori.  Thurston \cite{Th} introduced the deformation variety
$D(\T)$ parameterizing (incomplete) hyperbolic structures on
$\mathrm{int}(M)$ where each tetrahedron in $\T$ has the shape of some
honest ideal tetrahedron in $\H^3$.  The deformation variety plays a
key role in understanding hyperbolic Dehn filling \cite{Th, NZ}, and
is closely related to the $\PSL{2}{\C}$-character variety of $\pi_1(M)$.
Via the latter picture, ideal points of $D(\T)$ often give rise to
essential surfaces in $M$, and spun-normal surfaces are the natural
way to understand this process.  In this section, we sketch the needed
properties of $D(\T)$ from the point of view of \cite{Dapp,Till2}
which contain the omitted details.

\begin{figure}
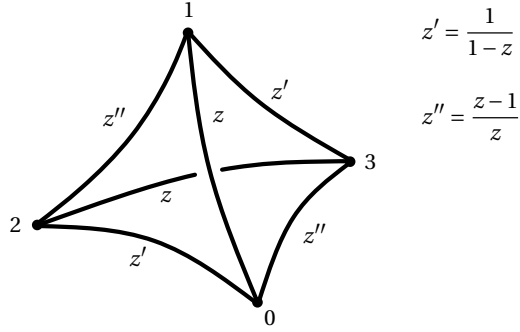

  \begin{cxyoverpic}{(153,133)}{scale=0.8}{images/ideal}
    ,(106,2)*+!UL{0}
    ,(150,69)*++!L{3}
    ,(2,39)*++!R{2}
    ,(74,130)*++!D{1}
    ,(124,43)*+!UL{z''}
    ,(48,82)*+!DR{z''}
    ,(58,32)*+!UR{z'}
    ,(109,95)*+!DL{z'}
    ,(78,90)*++!L{z}
    ,(50,65)*+++!UL{z}
    ,(180,130)*+!L{\displaystyle z' = \frac{1}{1 - z}}
    ,(180,90)*+!L{\displaystyle z'' = \frac{z -1}{z}}
  \end{cxyoverpic}
  \caption{The relationship between the shape parameters of the edges
    of an oriented tetrahedron in $\H^3$.  Our convention here agrees
    with \cite{SnapPy} and \cite[Section 2.2]{Till1}.}  \label{fig-edge-param}
\end{figure}

Suppose $\Delta$ is a non-degenerate ideal tetrahedron in $\H^3$, which
has an intrinsic orientation (i.e.~an ordering of its vertices).  Each
edge of $\Delta$ has a \emph{shape parameter}, defined as follows.  We
apply an orientation preserving isometry of $\H^3$ so that the
vertices of $\Delta$ are $(0, 1, \infty, z)$ and so this ordering induces the
orientation of $\Delta$.  The shape parameter of the edge $(0, \infty)$ is
then $z$, which lies in $\C \setminus \{0, 1\}$.  Opposite edges have the
same parameter, and any parameter determines all the others, as described in
Figure~\ref{fig-edge-param}, or as encoded in
\begin{equation}\label{eq-tet-shape}
 z'(1-z) = 1 \quad \mbox{and} \quad  z \, z' z'' = -1.
\end{equation}
Returning to our ideal triangulation $\T$, suppose it has $n$
tetrahedra. An assignment of hyperbolic shapes to all the tetrahedra
is given by a point in $(\C^3)^n$ which satisfies $n$ copies of the
equations (\ref{eq-tet-shape}).  The \emph{deformation variety}
$D(\T)$, also called the \emph{gluing equation variety}, is the
subvariety of possible shapes where we require in addition that the
\emph{edge equations} are satisfied: for each edge the product of the
shape parameters of the tetrahedra around it is $1$.  This requirement
says that the hyperbolic structures on the individual tetrahedra glue
up along the edge.

Because $D(\T) \subset \C^{3n}$ satisfies the conditions coming from
(\ref{eq-tet-shape}), at a point of $D(\T)$ no shape parameter takes
on a degenerate value of $\{0,1,\infty\}$.  Consequently, a point of
$D(\T)$ gives rise to a developing map from the universal cover
$\Mtil$ to $\H^3$ which takes each tetrahedron of $\widetilde \T$ to
one of the appropriate shape (see Lemma~\ref{lem-dev-exists} below).
This developing map is equivariant with respect to a corresponding
holonomy representation $\rho \maps \pi_1(M) \to \PSL{2}{\C}$.  In fact,
there is a regular map
\begin{equation}\label{eqn-map-to-char}
  D(\T) \to \Xbar(M) \quad \mbox{where $\Xbar(M)$ is the $\PSL{2}{\C}$-character variety of $\pi_1(M)$.}
\end{equation}
This map need not be onto, see e.g.~the last part of Section 10 of
\cite{Dapp}.  However, if $M$ is hyperbolic and no edge of $\T$ is
homotopically peripheral, then the image is nonempty.  In particular,
it contains a 1-dimensional irreducible component containing the
discrete faithful representation $\pi_1(M) \to \PSL{2}{\C}$ coming from
the unique oriented complete hyperbolic structure on $M$.

\begin{remark}\label{rem-edge-essen}
  In Lemma 2.2 of \cite{Till2}, the existence of
  (\ref{eqn-map-to-char}) is predicated on the edges of $\T$ being
  homotopically non-peripheral, whereas this condition is not
  mentioned in \cite{Dapp}.  Indeed it is not necessary to restrict
  $\T$, but as \cite{Dapp} is terse on this point, we give a proof here
  of:
  \begin{lemma}\label{lem-dev-exists}
    For any triangulation $\T$, a point in $D(\T)$ gives rise to a
    developing map $\Mtil \to \H^3$, and hence a holonomy
    representation $\pi_1(M) \to \PSL{2}{\C}$.   
  \end{lemma}
  In fact, the proof will show that if $D(\T)$ is nonempty, then a
  posteriori every edge in $\T$ is homotopically non-peripheral,
  meshing with Lemma 2.2 of \cite{Till2}.  A more detailed proof of a
  generalization of Lemma~\ref{lem-dev-exists} is given in
  \cite{TillmannSegerman}.
  \begin{proof}
    Let $N = M \setminus \partial M$ which we identify with the underlying space of
    $\T$ minus the vertices.  Looking that the universal cover of $N$,
    we seek a map
    \[
    d \maps \Ntil \to \H^3
    \] 
    which takes each ideal simplex in $N$ to an ideal simplex of $\H^3$
    with the assigned shape (in particular, we are not yet trying to
    define the map at infinity).  Let $\Ndot$ be $N$ minus the
    1-skeleton of $\T$, which deformation retracts to the dual
    1-skeleton of $\T$.  In particular, $\pi_1(\Ndot)$ is free, and
    universal cover $U$ of $\Ndot$ consists of tetrahedra with their
    1-skeletons deleted,  arranged so the dual 1-skeleton is an infinite
    tree.  Thus it is trivial to inductively define a map
    \[
    \dtil \maps U \to \H^3
    \]
    which takes what's left of each tetrahedron in $U$ to a correctly
    shaped ideal tetrahedron in $\H^3$ with its edges deleted.  Let
    $\Ntildot$ be $\Ntil$ minus the lifted 1-skeleton of $\T$.  Then we
    have covers $U \to \Ntildot \to \Ndot$.  The cover $\Ntildot \to \Ndot$
    corresponds to the normal subgroup $\Gamma$ of $\pi_1(\Ndot)$ generated
    by the boundaries of the dual 2-cells of $\T$, one corresponding to
    each edge.  The condition that the shape parameters have product 1
    for each edge mean that $\dtil$ is invariant under the deck
    transformation corresponding to the boundary of a dual 2-cell; hence
    $\dtil$ descends to a map $d$ of $\Ntildot = U/\Gamma$ to $\H^3$.  The
    same edge condition also means that $d$ extends over the deleted
    1-skeleton to the desired map $d \maps \Ntil \to \H^3$.  (This is
    perhaps easier to understand if one only wants the corresponding
    representation: the holonomy representation $\pi_1(\Ndot) \to
    \PSL{2}{\C}$ for $\dtil$ clearly has the boundary of each dual
    2-cell in its kernel, and thus factors through to a representation
    of $\pi_1(N)$.)
    
    Now that we have $d \maps \Ntil \to \H^3$ in hand, it is not hard
    to extend it to a continuous map from the end-compactification
    $\Nbar$ of $\Ntil$ to $\Hbar = \H^3 \cup S^2_\infty$.  This gives a
    pseudo-developing map in the sense of \cite[Section
    2.5]{Dunfield3}, and a posteriori certifies that the edges of $\T$
    are homotopically non-peripheral, since they go to infinite
    geodesics under $d$ which have two distinct limit points in
    $S^2_\infty$.
    \end{proof}
\end{remark}

\subsection{Ideal points and spun-normal surfaces}
\label{sub-ideal-spun}

We now describe the connection between $D(\T)$ and essential surfaces
in $M$, which has its genesis in the work of Culler and Shalen on the
character variety \cite{CS, CGLS}.  When $M$ has one boundary
component, a geometric component of $D(\T)$ has complex dimension one,
and it is common that all irreducible components of $D(\T)$ are also
curves.  Thus for simplicity we focus on an irreducible curve $D \subset
D(\T)$; for the full story of ideal points as points in Bergman's
logarithmic limit set, see \cite{Till2}.

As $D$ is an affine algebraic variety, it is not compact. Let $\Dtil$
be a smooth projective model for $D$, which in particular is a compact
Riemann surface together with a rational map $f \maps \Dtil \to D$
which is generically 1-1.  An \emph{ideal point} of $D$ is a point of
$\Dtil$ where $f$ is not defined.  For each edge of a tetrahedron in
$\T$, the corresponding shape parameter $z$ gives an
everywhere-defined regular function $z \maps \Dtil \to \BP^1(\C)$.
From (\ref{eq-tet-shape}), it is easy to see that at an ideal point,
the three shape parameters of a given tetrahedron are either $(z, z',
z'') = (0, 1, \infty)$ (or some cyclic permutation thereof), or all
take on values in $\C \setminus \{0,1\}$.

We next describe how to define from an ideal point $\xi$ of $D$ a
spun-normal surface $S(\xi)$.  For each tetrahedron $\Delta$ of $\T$,
we label each edge by the order of zero of the corresponding shape
parameter at $\xi$ (poles count as negative order zeros).  For
instance, in Figure~\ref{fig-edge-param}, if $z$ has a zero of order 2
at $\xi$, then the formulae for $z'$ and $z''$ mean that the edges of
$\Delta$ are labeled as shown in Figure~\ref{fig-tet-with-spun}(b).
In general, the labeling associated to $\xi$ similarly arises as the
edge shifts of a unique spun-normal picture in $\Delta$.  In the
case just mentioned, this is shown in
Figure~\ref{fig-tet-with-spun}(a); in general, if $n$ is the largest
order of zero of the shape parameters, then $S(\xi) \cap \Delta$ has
$n$ quads which are disjoint from the edges whose shape parameters are $1$ at
$\xi$.  That these local descriptions of $S(\xi)$ actually give a
spun-normal surface can be seen as follows.  Focus on an edge of $\T$,
and let $z_1, \ldots, z_k$ be the shape parameters of the tetrahedra
around it. Now on $D(\T)$ and hence on $\Dtil$ we have $\prod z_i =
1$, and taking orders of zeros turns this into the $Q$-matching
equation for that edge, namely that the sum of the shifts is 0.

Before addressing the question of when $S(\xi)$ is essential, we
mention that there is a closely related construction of Yoshida
\cite{Y} which also associates a surface to an ideal point of $D$; see
Segerman \cite{Seg} for the exact relationship between these two
surfaces.

\subsection{Ideal points and essential surfaces}
\label{sub-ideal-ess}

Culler and Shalen showed how to associate to an ideal point of the
character variety $\Xbar(M)$ an essential surface via a non-trivial
action on a tree \cite{CS}.  However, not every ideal point $\xi$ of
$D(\T)$ gives rise to an essential surface, as sometimes ideal points
of $D(\T)$ map to ordinary points of $\Xbar(M)$.  We now describe how
when $S(\xi)$ has non-empty boundary (in the sense of
Section~\ref{sec-ends-spun}) then it does come from an ideal point of
$\Xbar(M)$.

As this is the key condition, we first sketch how to determine
whether the surface $S(\xi)$ has nonempty boundary along a component $T$ of
$\partial M$ or instead consists of infinitely many boundary-parallel
tori; for details see \cite[Section 4]{Till2} and \cite[Sections 1 and
3]{Till1}.  For an element $\gamma \in \pi_1(\partial T)$, here is how
to calculate the intersection number between $\gamma$ and $\partial
S(\xi)$.  For a point $D(\T)$, the holonomy in the sense of
\cite{T} and \cite{NZ} is given by
\[
 h(\gamma) = z_1 z_2 \cdots z_k  \quad \mbox{for certain shape
   parameters $z_i$.}
\]
View the components of $\partial S(\xi)$ on $T$ as all oriented in the
same direction, which direction being determined by how $S(\xi)$ is
spinning out toward the boundary (see \cite[Section 3.1]{Till1}).
Then the algebraic intersection number of $\gamma$ and $\partial
S(\xi)$ is the order of zero of $h(\gamma)$ at $\xi$.  In particular,
by taking a basis for $\pi_1(T)$, it is easy to check whether $S(\xi)$
has boundary and, if so, what the slope is.

We now turn to the question of when $S(\xi)$ can be reduced to an
essential surface, in the following sense: a surface $S$ is said to
\emph{reduce} to $S'$ if there is a sequence of compressions, boundary
compressions, elimination of trivial 2-spheres, and elimination of
boundary-parallel components which turns $S$ into $S'$.  We then say
that $S'$ is a \emph{reduction} of $S$.  It will be convenient later
to consider more broadly spun-normal surfaces $S$ whose
$Q$-coordinates are a rational multiple of those of $S(\xi)$; we call
such $S$ \emph{associated} to $\xi$.

\begin{theorem}\label{thm-ideal-pt-to-surface}
  Let $\xi$ be an ideal point of a curve $D \subset D(\T)$.  Suppose a
  two-sided spun-normal surface $S$ associated to $\xi$ has non-empty
  boundary with slope $\alpha$ on a component $T$ of $\partial M$.
  Then any reduction of $S$ has nonempty boundary along $T$ with slope
  $\alpha$. In particular, $S$ can be reduced to a non-empty essential
  surface in $M$ which also has boundary slope $\alpha$.
\end{theorem}

\begin{proof}
  This will follow easily from \cite[Section 6]{Till2}, but to this
  end we note that we have defined ``spun-normal'' slightly
  differently than \cite{Till2}.  In particular, what we call
  spun-normal with non-empty boundary he calls simply spun-normal.
  Moreover, in \cite{Till2} the surface $S(\xi)$ is made two-sided
  simply by doubling its $Q$-coordinates if it's not; we adopt this
  convention for this proof.

  First, we reduce from an arbitrary $S$ associated to $\xi$ to
  $S(\xi)$ itself.  Let $S_0$ be the spun-normal surface corresponding
  to the primitive lattice point on the ray $\R_+ \cdot S$, i.e.~$S_0
  = (1/g) S$ where $g$ is the $\gcd$ of the coordinates of $S$.  If
  $S_0$ is two-sided, then both $S$ and $S(\xi)$ are simply a disjoint
  union of parallel copies of $S_0$, and thus we can focus on $S(\xi)$
  instead.  Should $S_0$ have a one-sided component, then as $S$ and
  $S(\xi)$ are two-sided, they are both integer multiples of $2 \cdot S_0$,
  and again we can focus on $S(\xi)$.

  We now relate $S(\xi)$ to the Bass-Serre tree associated to an ideal
  point of the $\PSL{2}{\C}$-character variety $\Xbar(M)$.  Following
  Section 5.3 of \cite{Till2}, we use $\TN$ to denote the simplicial
  tree dual to the spun-normal surface $S(\xi)$. (Unlike \cite{Till2},
  we require $S(\xi)$ to have infinitely many triangles in every
  corner of every tetrahedron; hence the dual tree to $S(\xi)$ is
  $\TN$ rather than the $\TS$ of Section 5.2 of \cite{Till2}.)  Let $N
  = \T \setminus \T^0 \cong M \setminus \partial M$, and let $p \maps
  \Ntil \to N$ be the universal covering map. (Note: our $N$ is
  called $M$ in \cite{Till2}.) There is an equivariant map $f \maps
  \Ntil \to \TN$ where the preimage of the midpoints of the edges in
  $\TN$ is precisely $p^{-1}(S(\xi))$.

  Now fix a simple closed curve $\beta \in \pi_1(T)$ which intersects
  $\alpha$ exactly once.  As discussed above, we can orient
  $\beta$ so that the holonomy $h(\beta)$ has a pole at $\xi$.  By
  Proposition 6.10 of \cite{Till2}, there is an associated ideal point
  $\xi'$ of a curve in $\Xbar(M)$ so that there is a
  $\pi_1(M)$-equivariant map from $\TN$ to the simplicial tree
  $T_{\xi'}$ associated to $\xi'$.  In particular, since $h(\beta)$
  has a pole, the action of $\beta$ on $T_{\xi'}$ is by a
  fixed-point free loxodromic transformation.  Because of the map $\TN
  \to T_{\xi'}$, it follows that $\beta$ also acts on $\TN$ by a
  loxodromic.

  As in Section~\ref{sec-ends-spun}, we identify $M$ with a suitable
  subset of $N$, and henceforth abuse notation by denoting $S \cap M$
  by $S$.  By restricting the domain, we get that $S$ is dual to the
  equivariant map $f \maps \Mtil \to \TN$.  Now if $S'$ is a reduction
  of $S$, we can modify $f$ so that that $S'$ is still dual to $\TN$.
  If $T \cap S'$ were empty, it follows that $\pi_1(T)$ acts on $\TN$
  with a global fixed point. Thus since $\beta$ acts on $\TN$ as a
  loxodromic, we have that $S'$ has nonempty boundary along $T$, as
  claimed.
\end{proof}

\section{Ideal points of varieties of gluing equation type }
\label{sec-ideal-gluing}

In this section, we consider a class of complex algebraic varieties
that arise from the deformation varieties of the last section by
focusing on a single shape  parameter for each tetrahedron.  Such
varieties were first considered by Thurston \cite{Th} and
Neumann-Zagier \cite{NZ}.

We start with a subgroup $\Lambda \subset \Z^{2n+1}$, which we call a
\emph{lattice} even when its rank is not maximal.  Let $\Coo = \C \setminus
\{0,1 \}$, and consider the variety $V(\Lambda) \subset \left(\Coo\right)^n$ of
points satisfying
\begin{equation}\label{eq-def-var}
  z_1^{a_1}z_2^{a_2} \cdots z_n^{a_n} (1 - z_1)^{b_1}(1-z_2)^{b_2} \cdots (1 - z_n)^{b_n}  = (-1)^c 
\end{equation}
for all $(a_1,\ldots,a_n,b_1,\ldots,b_n, c) \in \Lambda$.  Since $z_i$ and $(1-z_i)$
are never $0$ for $z_i \in \Coo$, these equations always make sense
even when some $a_i$ or $b_i$ is negative.  Henceforth, we assume that
$\rank(\Lambda) \leq n-1$, and call such a $V(\Lambda)$ a \emph{variety of gluing
  equation type}.  In the final application, the variety $V(\Lambda)$ will be a
complex curve and hence $\rank(\Lambda) = n - 1$. 

\begin{remark}
  Replacing the lattice $\Lambda$ with an arbitrary subset $\Omega$ of
  $\Z^{2n+1}$ doesn't broaden this class of examples, since $V(\Omega) =
  V\big(\mathrm{span}\pair\Omega\big)$.  Conversely, when testing whether
  a point is in $V(\Lambda)$, it suffices to consider only the finitely
  many equations coming from a given $\Z$-basis for $\Lambda$.  More
  precisely, let $M(\Lambda)$ be a matrix whose $r$ rows are a basis for $\Lambda$,
  and write it as
  \begin{equation}
    \label{eq-MLambda}
    M(\Lambda) = \left(\begin{array}{ccc|ccc|c}
        & & & & & & \\
        &A & & & B & &c \\
        & & & & & & \end{array}\right) 
  \end{equation}
  where $A$ and $B$ are $r \times n$ matrices, and $c$ is an $r \times 1$
  column vector. Then $V(\Lambda)$ can be described by \eqref{eq-def-var}
  for all rows $(a_1,\dots,a_n,b_1,\dots,b_n,c)$ of the matrix
  $M(\Lambda)$.
\end{remark}

\begin{example}\label{ex-DT}
  As in Section~\ref{sec-def-var}, suppose $\T$ is an ideal
  triangulation of a manifold $M$, and consider its deformation
  variety $D(\T) \subset \C^{3n}$, where $n$ is the number of tetrahedra in
  $\T$.  If we fix a preferred edge in each tetrahedron, then its shape
  parameter $z_i$ determines $z'_i$ and $z''_i$ as noted in
  Figure~\ref{fig-edge-param}; using these expressions for $z'_i$ and
  $z_i''$ turns each edge equation into one of the form
  (\ref{eq-def-var}).  Thus projecting away the other coordinates
  gives an injection $D(\T) \hookrightarrow (\Coo)^n$, and the image variety 
  $V$ is given by 
  $V(\Lambda)$ for some $\Lambda$.  The number of edges of $\T$ is equal to $n$,
  but if $D(\T)$ is non-empty then we argue that the rank of $\Lambda$ is
  $n - k$, where $k$ is the number of components of $\partial M$.

  First, the matrix $M(\Lambda)$ minus its last column has rank $r = n
  - k$; this is Proposition 2.3 of \cite{NZ} when $M$ is hyperbolic,
  and Theorem 4.1 and remark following it in \cite{Neumann1992} for the
  general case.  Thus $\Lambda$ has a basis where $r$ of the vectors have
  nonzero $a$ or $b$ components, and the rest have only the $c$
  component being nonzero.  Since $D(\T)$ is assumed nonempty, all of
  the latter must correspond to the equation $1 = 1$ rather than
  $1 = -1$ and hence may be omitted.  Thus $V$ is defined by a lattice
  $\Lambda$ of rank $r$.  As $r \leq n - 1$, the projection $V$ of $D(\T)$ is
  indeed a variety of gluing equation type.
\end{example}

\begin{remark}  
  F. Rodriguez Villegas pointed out to us that $V(\Lambda)$ is the
  intersection of a toric variety with an affine subspace.  Precisely,
  if $\C^* = \C \setminus \{0\}$ then it is isomorphic to the subspace of
  $(\C^*)^{2n+1}$ cut out by
  \begin{equation}
    \label{eq-def-toric}
    z_1^{a_1}z_2^{a_2} \cdots z_n^{a_n} w_1^{b_1}w_2^{b_2} \cdots w_n^{b_n} u^c = 1 \quad \mbox{for all $(a_1,\ldots,a_n,b_1,\ldots,b_n, c) \in \Lambda$.}
  \end{equation}
  together with $u = -1$ and $z_i + w_i = 1$ for $1 \leq i \leq n$.
  This seems potentially very useful, though we do not exploit it
  here.
\end{remark}

\subsection{Ideal points}

Let $\C^* = \C \setminus \{0\}$ and $\Co = \C \setminus \{ 1\}$.  Now in $(\Co)^n$,
consider the closure $\Vbar(\Lambda)$ of $V(\Lambda)$ in (equivalently) either
the Zariski or the analytic (naive) topology.  Points of $\Vbar \setminus V$
will be called \emph{ideal points}.  In the context of
Example~\ref{ex-DT} and Section~\ref{sub-ideal-spun}, these are images
of ideal points $\xi$ of $D(\T)$ where the preferred shape parameters
are either 0 or nondegenerate at $\xi$.  By choosing the shape
parameters appropriately, any ideal point of $D(\T)$ gives an ideal
point of the corresponding $V(\Lambda)$.  The individual ideal points of
$D(\T)$ can be found by analyzing the local structure, typically
highly singular, of the ideal points of the $V(\Lambda)$.

So returning to the context of a general $V = V(\Lambda)$, we seek to
understand the local structure of $\Vbar$ near an ideal point $p$.  In
particular, we need to find a holomorphic map from the open unit disc
$D \subset \C$ of the form
\begin{equation}\label{eq-hol-param}
  f \maps (D, 0) \to (\Vbar, p ) \quad \mbox{where $f(D \setminus \{0\}) \subset V$.}
\end{equation}
Taking $t$ as the parameter on $D$, we have 
\begin{equation}\label{eq-param-details}
  z_i = t^{d_i} u_i (t)  
\end{equation}
where $d_i \geq 0$ and $u_i$ are holomorphic functions on $D$ with
$u_i(0) \neq 0$ for all $i$, and $u_i(0) \neq 1$ when $d_i = 0$.  As
always, each $u_i$ can be represented by a convergent power series in
$\C [[t]]$.

The lattice $\Lambda$ constrains the possibilities for $\dd =
(d_1,d_2,\ldots,d_n)$ as follows.  Consider the equations coming from a
matrix $M(\Lambda)$ as in \eqref{eq-MLambda}, and substitute
(\ref{eq-param-details}) into (\ref{eq-def-var}).  If we send $t \to
0$, it follows that $\dd$ is in $\ker(A)$.  This motivates:
\begin{definition}
  A \emph{degeneration vector} is a nonzero element $\dd \in \ker(A) \cap
  (\Z_{\geq 0})^n$.  It is \emph{genuine} if it arises as in
  (\ref{eq-param-details}) for some ideal point of $V(\Lambda)$.
\end{definition}

\begin{remark}\label{rem-deg-spun}
  If $V$ comes from $D(\T)$ as discussed in Example~\ref{ex-DT}, then
  degeneration vectors correspond precisely to the $Q$-coordinates of
  certain spun-normal surfaces as follows.  In a tetrahedron with a
  preferred shape parameter $z$, we say the \emph{preferred quad} is
  the one with shift $+1$ along the preferred edge; equivalently, the
  preferred quad of the tetrahedron labeled as in
  Figure~\ref{fig-edge-param} is shown in
  Figure~\ref{fig-tet-with-spun}(a).  Now, in the notation of
  Section~\ref{sec-spun-normal}, consider the face $C'$ of $C(\T)$
  where all non-preferred quads have weight zero.  The relationship
  described in Section~\ref{sub-ideal-spun} between edge equations and
  $Q$-matching equations shows that if we focus on the subspace of
  preferred quads, the $Q$-matching equations are simply given by the
  $A$ part of the $M(\Lambda)$ matrix.  Thus degeneration vectors are
  precisely the integer points of $C'$, and each corresponds to a
  spun-normal surface.  So when $d$ is genuine, it is the
  $Q$-coordinates of a spun-normal surface $S(\dd)$ associated to an
  ideal point $\xi$ of $D(\T)$.  (Technical aside: we have not insisted
  that $f$ in (\ref{eq-hol-param}) is generically $1-1$, thus $d$ may
  be an integer multiple of the vector of the orders of zero of the
  $z$ at the corresponding ideal point $\xi$.  Hence, $S(\dd)$ may be
  some integer multiple of $S(\xi)$.) 
  
\end{remark}

Thus the key question for us here is when a given degeneration
vector is genuine.  The following is the main technical tool from
\cite{Ka}, and underlies our Theorems~\ref{thm-one-cusp} and
\ref{thm-filling}:
\begin{proposition}\label{prop-Kabaya-key}
  Suppose a degeneration vector $\dd$ is totally positive, i.e.~each
  $d_i > 0$.  If $A$ has rank $n-1$, then $\dd$ is genuine.  
\end{proposition}
\noindent
We include a detailed proof of this in our current framework, as part
of a more general discussion of which degeneration vectors are genuine.   

\subsection{Genuine degeneration vectors}

Fix a degeneration vector $\dd$ which we wish to test for being
genuine.  For convenience, we reorder our variables so that $d_i = 0$
for precisely $i \geq k > 1$.  Taking our lead from the substitution in
\eqref{eq-param-details}, and arbitrarily folding $u_1$ into $t$, we consider
\[
 \pi \maps \C^n \to \C^n \mtext{given by} (t, u_2, \ldots, u_n)  \mapsto (t^{d_1}, t^{d_2} u_2, \ldots, t^{d_n} u_n)
\]
We set $W(\Lambda,\dd)$ to be the preimage of $V$ under $\pi$, regarded
as a subvariety of 
\[
U = \pi^{-1}\big(  (\Coo)^n \big) = (\C^*)^{n}  \setminus \{ t^{d_1} = 1, t^{d_i} u_i = 1\}
\]
Equivalently, using \eqref{eq-param-details}, we see $W(\Lambda, \dd)$ is the subset
of $U$ cut out by 
\begin{equation}\label{eq-def-var2}
  u_2^{a_2} \cdots u_n^{a_n} (1 - t^{d_1})^{b_1}(1-t^{d_2}u_2)^{b_2} 
\cdots (1 - t^{d_n}u_n)^{b_n}  = (-1)^c 
\end{equation}
for $(a, b, c) \in \Lambda$.  To examine whether $\dd$ is genuine, we need
to allow $t$ to be zero.  So consider
\[
\Ubar = \C \times (\C^*)^{n - 1}  \setminus \{ t^{d_1} = 1, t^{d_i} u_i = 1\}
\]
and let $\Wbar(\Lambda, \dd)$ be the closure of $W(\Lambda, \dd)$ in $\Ubar$.
Defining
\[
  \Wbar_0(\Lambda, \dd) =  \Wbar(\Lambda, \dd) \cap \{t = 0\}, 
\]
we have a simple test for when $\dd$ is genuine:
\begin{lemma}\label{lem-gen-char}
  If $\dd$ is genuine, then $\Wbar_0(\Lambda, \dd)$ is nonempty.  Almost
  conversely, if $\Wbar_0(\Lambda, \dd)$ is nonempty then a positive
  integer multiple of $\dd$ is genuine.
\end{lemma}
The reader whose focus is on Theorems~\ref{thm-one-cusp} and
\ref{thm-filling} may skip the proof of Lemma~\ref{lem-gen-char}, as
the proof of Proposition~\ref{prop-Kabaya-key} does not depend on it.
\begin{proof}
  First suppose that $\dd$ is genuine.  Consider the analytic
  functions $u_i(t)$ in \eqref{eq-param-details}; by replacing $t$
  with $t \big(u_1(t)\big)^{-1/d_1}$, which is analytic near $t=0$, we
  may assume $u_1(t)$ is the constant function $1$.  Now, for small $t
  \neq 0$ the function
  \[
  t \mapsto \big(t, u_2(t), u_3(t), \ldots,u_n(t)\big)
  \]
  has image contained in $W(\Lambda, \dd)$.  Thus by continuity, the point $\big(0,
  u_2(0), \ldots, u_n(0)\big)$ is in $\Wbar_0(\Lambda, \dd)$, as
  needed.

  Now suppose instead $p$ is a point of $\Wbar_0(\Lambda, \dd)$.  Dropping
  $\Lambda$ and $\dd$ from the notation, we argue it is enough to show
  \begin{claim}
    There is an irreducible curve $C \subset \Wbar$ containing $p$ on which
    $t$ is nonconstant.
  \end{claim}
  If the claim holds, let $\Ctil$ be a smooth projective model for
  $C$, with $f \maps \Ctil \to C$ the corresponding rational map.  If
  we take $s$ to be a holomorphic parameter on $\Ctil$ which is $0$ at
  some preimage of $p$, then $\pi \circ f \circ s$ shows that $m \cdot \dd$ is a
  genuine degeneration vector, where $m > 0$ is the order of zero of the
  $t$-coordinate of $f$ at $s = 0$.  
  
  To prove the claim, let $Y$ be an irreducible component of $\Wbar$
  containing $p$.  Since $\Wbar$ was defined by taking the closure of
  $W$ in $\Ubar$, it follows that $t$ is nonconstant on $Y$.  If $j =
  \dim Y > 1$, we will construct an irreducible subvariety $Y'$ of
  dimension $j-1$ which contains $p$ and on which $t$ is nonconstant.
  Repeating this inductively will produce the needed curve $C$.

  As $Y$ is irreducible, and $Y_0 = Y \cap \{ t= 0 \}$ is a nonempty
  proper algebraic subset, it follows that $\dim Y_0 = j - 1$.  There
  are coefficients $\alpha_i \in \C$ so that the polynomial
  \[
  g = \alpha_1 + \alpha_2 u_2 + \alpha_3 u_3 + \cdots + \alpha_n u_n
  \]
  is nonconstant on \emph{every} irreducible component of $Y_0$, and
  where $g(p) = 0$.  (If we temporarily view $p$ as the origin of our
  coordinate system, then any linear functional whose kernel fails to
  contain the linear envelope of any component of $Y_0$ works for
  $g$.)  Now set $Y' = Y \cap \{g = 0\}$, which contains $p$ and has
  dimension $j - 1$ as $g$ is nonconstant on $Y$.  Moreover $Y'
  \cap \{ t = 0 \} = Y_0 \cap \{ g = 0 \}$ has dimension $j - 2$ as $g$ is
  nonconstant on every component of $Y_0$.  Thus an irreducible
  component of $Y'$ containing $p$ has dimension $j -1$ and $t$ is
  nonconstant on it, as needed.
\end{proof}

\subsection{The first-order system}

Suppose that $\beta = (0, \beta_2, \beta_3,\ldots, \beta_n)$ is a point of $\Wbar(\Lambda,
\dd)$.  Substituting $t = 0$ into \eqref{eq-def-var2} we get that $\beta$ satisfies
\begin{equation}\label{eq-first-order}
  \beta_2^{a_2} \cdots \beta_n^{a_n} (1-\beta_k)^{b_k} \cdots (1-\beta_n)^{b_n} = (-1)^c \quad \mbox{for all $(a; b; c) \in \Lambda$.}
\end{equation}
We call the union of all such equations, together with $t=0$, the
\emph{first-order system}, and denote the corresponding subset of
$\{0\} \times (\C^*)^{k-2} \times (\Coo)^{n-k+1}$ by $W_0(\Lambda,
\dd)$.  Notice that $W_0(\Lambda, \dd)$ contains $\Wbar_0(\Lambda,
\dd)$, but is not a priori equal to it, as the latter may contain
points which are not in the closure of $W(\Lambda, d)$.  As the former
is easier to work with in practice, we show
\begin{lemma}
  \label{thm.genuine2} 
  Suppose $W_0(\Lambda, \dd)$ is nonempty and has dimension 0.  Then some
  multiple of $\dd$ is genuine.
\end{lemma}

As we discuss later, in small examples this condition is easy to check
using Gr\"obner bases.  As with Lemma~\ref{lem-gen-char}, on which it
depends, it is not actually used to prove
Proposition~\ref{prop-Kabaya-key}.

\begin{proof}
  Consider the subvariety $\Wtil$ of $\Ubar$ cut out by the equations
  \eqref{eq-def-var2} coming from the $r$ rows of a fixed matrix
  $M(\Lambda)$ defining our original variety $V(\Lambda)$.  Then $\Wtil$
  contains both $W_0(\Lambda, \dd)$ and $\Wbar(\Lambda, \dd)$.  Let $p$ be a
  point of $W_0(\Lambda, \dd)$, and $Y$ an irreducible component of $\Wtil$
  containing $p$.  As $\Wtil$ is defined by $r \leq n-1$ equations and $\dim
  \Ubar = n$, the variety $Y$ must have dimension at least 1.  As $Y
  \cap \{t = 0\}$ is contained in the finite set $W_0(\Lambda, \dd)$, it
  follows that all but finitely many points of $Y$ are in $W(\Lambda,
  \dd)$.  Hence $p \in \Wbar(\Lambda, \dd)$, and Lemma~\ref{lem-gen-char}
  implies that a multiple of $\dd$ is genuine.
\end{proof}

We now have the needed framework to show Proposition~\ref{prop-Kabaya-key}.

\begin{proof}[Proof of Proposition~\ref{prop-Kabaya-key}]
  Let $d$ be a totally positive degeneration vector.  By hypothesis,
  the submatrix $A$ of $M(\Lambda)$ has rank $n-1$, and so in particular
  $M(\Lambda)$ has $n-1$ rows.  We reorder the variables so that the matrix
  $A'$ gotten by deleting the first column of $A$ also has rank $n-1$.

  To show $d$ is genuine, we start by examining the solutions $W_0(\Lambda,
  \dd)$ to the first-order equations.  As all $d_i > 0$, these
  equations are simply $t=0$ and 
  \begin{equation}\label{eq-simple-first-order}
    \beta_2^{a_2} \beta_3^{a_3} \cdots \beta_n^{a_n}  = (-1)^c \quad \mbox{for all $(a; b; c) \in \Lambda$.}
  \end{equation}
  where we require each $\beta_i \in \C^*$.  Note that any solution in
  $\C^{n-1}$ to the linear equations
  \begin{equation}\label{eq-linearized-first-order}
    a_2 x_2 + a_3 x_3 + \cdots + a_n x_n = c \pi i \quad \mbox{for all $(a; b; c) \in \Lambda$}
  \end{equation}
  gives rise to one of (\ref{eq-simple-first-order}) via the map
  $\C^{n-1} \to (\C^*)^{n-1}$ which exponentiates each coordinate.
  Since $\rank(A') = n -1$, the equations
  (\ref{eq-linearized-first-order}) have a solution and hence so
  do (\ref{eq-simple-first-order}).

  We will use the inverse function theorem to show that $d$ is
  genuine.  To set this up, let $\Wtil$ be the subvariety of $\C^n$
  with coordinates $(t, u_2, u_3,\ldots,u_n)$ cut out by the $n-1$
  equations (\ref{eq-def-var2}) coming from rows of the matrix $M(\Lambda)$.
  Fix a point $\beta \in W_0(\Lambda, \dd) \subset \Wtil$, and let $J$ be the $(n-1) \times
  n$ Jacobian matrix of these equations at $\beta$.  Let $J'$ be the
  submatrix of $J$ gotten by deleting the first column (which
  corresponds to $\partial/\partial t$).  If $J'$ has rank $n-1$, then the inverse
  function theorem implies that $\Wtil$ is a smooth curve at $\beta$.
  Moreover, this curve is transverse to $\{ t = 0 \}$ since $\rank(J') =
  n - 1$ forces any nonzero element of $\ker(J) = T_p\Wtil$ to have
  nonzero first component.

  Thus it remains to calculate the matrix $J'$.  As all $d_i > 0$, taking
  $\partial/\partial u_i$ of (\ref{eq-def-var2}) at $\beta$ gives $ a_i (-1)^c /
  \beta_i$.  Thus the columns of $J'$ are nonzero multiples of those of
  $A'$, and hence $\rank(J') = \rank(A') = n - 1$ as needed.  Thus $\dd$ is genuine.
\end{proof}

\subsection{Examples} 

Both hypotheses of Proposition~\ref{prop-Kabaya-key} are 
necessary, even for the weaker conclusion that the first-order
equations have a solution.  Here are two examples with $V(\Lambda) \neq 0$
which illustrate this.

First, for $n = 2$ consider the span $\Lambda$ of $(0,1;\ 1, -1;\ 1)$;
here, $V(\Lambda)$ is given by a single equation
\begin{equation}\label{eq-ex-1}
\frac{z_2(1 - z_1)}{1 - z_2} = -1
\end{equation}
which defines the nonempty plane conic $z_1 z_2 = 1$.  For the
degeneration vector $\dd = (1,0)$, the first-order system is
\[
\frac{\beta_2}{1 - \beta_2} = -1,
\]
which is equivalent to $0= -1$ and hence has no solutions. 
So $\dd$ is not genuine, even though $A = (0,
1)$ has maximal rank. This shows the total positivity of $\dd$ is
necessary for Proposition~\ref{prop-Kabaya-key}.

Second, again for $n = 2$, consider the span $\Lambda$ of $(0,0;\ 1, 1;\
-1)$.  Then $V(\Lambda)$ is again a nonempty plane conic, and is given by
\begin{equation}\label{eq-ex-2}
  (1 - z_1)(1 - z_2) = -1.
\end{equation}
Here, any $\dd$ is a degeneration vector since $A = (0,0)$,  so take
$\dd = (1,2)$.  Then the first-order system is simply $1 = -1$ which
has no solutions.  So $\dd$ is not genuine, even though $\dd$ is
totally positive. This shows that the condition that $\rank(A)$ is
maximal is also necessary for Proposition~\ref{prop-Kabaya-key}.

\section {Proof of Theorem~\ref{thm-one-cusp}}
\label{sec-proof-one-cusp}

In this section, we prove 
\begin{mainonebackwards}
  Let $\T$ be an ideal triangulation of a compact oriented \3-manifold
  $M$ with $\partial M$ a torus.  Suppose $S$ is a vertex spun-normal
  surface in $\T$ with non-trivial boundary.  If $S$ has a
  quad in every tetrahedron of $\T$, then $S$ is essential.
\end{mainonebackwards}
The requirement that $\partial M$ is a single torus, rather than
several, is simply for notational convenience; the proof works
whenever $S$ has at least one non-trivial boundary component.

We first rephrase Theorem~\ref{thm-one-cusp} in the form in which we
will prove it.  Throughout this section, let $\T$ be an ideal
triangulation as in Theorem~\ref{thm-one-cusp}.  Recall from
Section~\ref{sec-spun-normal} if $\T$ has $n$ tetrahedra, then the
$Q$-coordinates of a spun-normal surface are given by a vector in
$\R^{3n}$ that lives in the linear subspace $L(\T)$ of solutions to
the $Q$-matching equations.  Specifically, each spun-normal surface
gives an integer vector in the convex cone $C(\T) = L(\T) \cap \R_+^{3
  n}$.

Suppose we fix a preferred type of quad in each tetrahedron; such a
choice will be denoted by $Q$.  Let $\R_Q^n \subset \R^{3n}$ be the
corresponding subspace where all non-preferred quads have weight $0$.
Define $L(\T, Q) = L(\T) \cap \R^n_Q$ and $C(\T,Q) = C(\T) \cap \R^n_Q$.
We will show the following:
\begin{theorem}\label{thm-one-cusp-tech}
  Suppose $S$ is a spun-normal surface with non-empty boundary which
  has a quad in every tetrahedron. Let $Q$ be the corresponding quad
  type.  If $\dim L(\T, Q) = 1$, then $S$ is essential.
\end{theorem}

\begin{proof}[Proof of Theorem~\ref{thm-one-cusp} from Theorem~\ref{thm-one-cusp-tech}]
  Let $S$ be a vertex spun-normal surface which has a quad in every
  tetrahedron; we need to show that $\dim L(\T, Q) = 1$.  Since $C(\T) =
  L(\T) \cap \BR_+^{3n}$, a face (of any dimension) of $C(\T)$
  corresponds to setting some subset of the coordinates to 0.  Thus
  since $S$ is a vertex solution, there are coordinates $u_i$ so
  that $C(\T) \cap \{u_1 = \cdots = u_k = 0\}$ is the ray $\BR_+ \cdot S$.  Let
  $Q$ be the unique quad type compatible with $S$.  As $S$ has nonzero
  weight on every quad in $Q$, we must have 
  \begin{equation}\label{eq-vertex}
    \BR_+ \cdot S  = C(\T) \cap \setdef{u_i = 0}{u_i\not\in Q} = C(\T) \cap \R^n_Q = C(\T, Q)
  \end{equation}
  Next we argue that $C(\T, Q) = L(\T, Q) \cap \R^{3n}_+$ has the same
  dimension as $L(\T,Q)$ itself.  This follows since all
  $\R^n_Q$-coordinates of $S$ are positive, and thus all nearby points
  to $S$ in $L(\T,Q)$ are also in $C(\T,Q)$.  Thus $\dim C(\T, Q) =
  \dim L(\T, Q)$.  As $\dim C(\T, Q) = 1$ by (\ref{eq-vertex}), the
  fact that $\dim C(\T, Q) = \dim L(\T, Q)$ shows that the hypotheses
  of Theorem~\ref{thm-one-cusp} imply those of
  Theorem~\ref{thm-one-cusp-tech}.  (In fact, the hypotheses of the
  two theorems are equivalent.)
\end{proof} 

We break the proof of Theorem~\ref{thm-one-cusp-tech} into two lemmas. 
\begin{lemma}\label{lem-comes-from-ideal}
  Suppose $S$ is a spun-normal surface with a quad in every
  tetrahedron.  Suppose that $\dim L(\T, Q) = 1$ for the quad type $Q$
  determined by $S$.  Then there is an ideal point $\xi$ of $D(\T)$ so
  that $S$ is associated to $\xi$.
\end{lemma}
\begin{lemma} \label{lem-act-essen} Suppose $S$ is a connected,
  two-sided, spun-normal surface with a quad in every tetrahedron.
  Suppose that $\dim L(\T, Q) = 1$ for the quad type $Q$ determined by
  $S$.  If every reduction of $S$ has nonempty boundary, then $S$ is
  essential.
\end{lemma}
We establish these lemmas below after first deriving the
theorem from them.

\begin{proof}[Proof of Theorem~\ref{thm-one-cusp-tech}]
  First, we reduce to the case that $S$ is two-sided and connected.
  Let $S_0$ be the spun-normal surface corresponding to the primitive
  lattice point on the ray $\R_+ \cdot S$, i.e.~$S_0 = (1/g) S$ where
  $g$ is the $\gcd$ of the coordinates of $S$.  The surface $S_0$ must
  be connected, since if not it would be the sum of two surfaces in
  $C(\T, Q)$ which is just $\R_+ \cdot S$ since $\dim L(\T, Q) = 1$.
  Now, the surface $S$ is essential if and only if $S_0$ is, so we shift
  focus to $S_0$.  If $S_0$ is one-sided, then by definition $S_0$ is
  essential if and only if $2 \cdot S_0$ is, and we focus on the latter
  (which is still connected).  Thus we have reduced to the case that
  $S$ is connected and two-sided.

  Now by Lemma~\ref{lem-comes-from-ideal}, there is an ideal point
  $\xi$ of $D(\T)$ so that $S$ is associated to $\xi$.  By
  Theorem~\ref{thm-ideal-pt-to-surface}, the surface $S$ can be
  reduced to a nonempty essential surface $S'$ with nonempty boundary.
  By Lemma~\ref{lem-act-essen}, the surface $S = S'$ and $S$ is
  essential, as required.
\end{proof}

\begin{proof}[Proof of Lemma~\ref{lem-comes-from-ideal}]
  In each tetrahedron $\Delta$ of $\T$, focus on the edge which has
  shift +1 with respect to the quad that is in $S$.  By
  Example~\ref{ex-DT}, if we focus solely on the corresponding shape
  parameters, this expresses the deformation variety $D(\T)$ as a
  variety $V(\Lambda)$ of gluing equation type.  Moreover, as
  discussed in Remark~\ref{rem-deg-spun}, the degeneration vectors
  $\dd$ of $V(\Lambda)$ correspond precisely to the spun-normal
  surfaces in $C(\T, Q)$.  Indeed, the $Q$-matching equations cutting
  out $L(\T, Q)$ from $\R_Q^n$ are equivalent to those given by the
  $A$ submatrix of $M(\Lambda)$.

  Let $\dd$ be the degeneration vector corresponding to the surface
  $S$.  By hypothesis $\dim L(\T,Q) = 1$, and so by the
  connection above we know $\rank(A) = n - 1$.  Thus by
  Proposition~\ref{prop-Kabaya-key}, the degeneration vector $\dd$ is
  genuine, and so by Remark~\ref{rem-deg-spun} the surface $S$ is
  associated to some ideal point $\xi$ of $D(\T)$, as needed.
\end{proof}
  
Before proving Lemma~\ref{lem-act-essen}, we sketch the basic idea,
which was suggested to us by Saul Schleimer and Eric Sedgwick.  If $S$
compresses, do so once to yield a surface $S'$ which is disjoint from
$S$.  Now normalize $S'$ to $S''$; while this may result in additional
compressions, the surface $S''$ is nonempty by hypothesis.  The
original normal surface $S$ acts as a barrier during the normalization
of $S'$ \cite{Rubinstein1997}, and so $S''$ is disjoint from $S$.  Thus the
quads in $S''$ are compatible with those of $S$.  Now as $\dim L(\T,Q)
= 1$, we must have that $S = S''$ and so the initial compression was
trivial and hence $S$ is essential. 

If $S$ was an ordinary (non-spun) normal surface, this sketch would
essentially be a complete proof.  Unfortunately, the spun-normal case
introduces some additional technicalities, particularly as we are not
assuming that $M$ is hyperbolic, and hence we can't appeal directly to
\cite{Walsh} to ensure that $S'$ can be normalized at all. 

\begin{proof}[Proof of Lemma~\ref{lem-act-essen}]
  As in Section~\ref{sec-ends-spun}, we pick a neighborhood $\Nv$ of
  the vertex $\v$ of $\T$ so that $S$ meets the torus $\Bv = \partial
  \Nv$ in nonseparating curves with consistent canonical orientations.
  We now identify $M$ with $\T \setminus \mathrm{int}(\Nv)$.   Except
  for the very end of the proof, we will focus on $S \cap M$ and so
  denote it simply by $S$.

  If $S$ is not essential as the lemma claims, there are three
  possibilities:
  \begin{enumerate} 
  \item \label{en-compresses} $S$ has a
    genuine compressing disc $D$.
  \item \label{en-bdry-comp} $S$ is incompressible but has a genuine
    $\partial$-compression $D$.
  \item \label{en-bdry-parallel} $S$ is boundary parallel.  
  \end{enumerate}
  Case (\ref{en-bdry-parallel}) is ruled out since $S$ can be reduced
  to an essential surface.  In case (\ref{en-bdry-comp}), consider the
  arc $\alpha = D \cap \partial M$.  If the end points of $\alpha$ are
  on the same component of $\partial S$, then the incompressibility of
  $S$ forces $D$ to be a trivial $\partial$-compression.  When instead
  $\alpha$ joins two components of $\partial M$, incompressibility
  means that the connected surface $S$ is an annulus.  But then
  compressing $S$ along $D$ gives a disc $S'$ whose boundary is
  inessential in $\partial M$.  This contradicts that that every
  reduction of $S$ yields a surface with nonempty boundary.

  Thus it remains to rule out (\ref{en-compresses}).  Now let $D$ be a
  compressing disc for $S$.  Compress $S$ along $D$ and slightly
  isotope the result to yield a surface $S_1$ disjoint from $S$.  Now
  further compress and otherwise reduce $S_1$ in the complement of $S$
  to give a surface $S_2$ which is disjoint from $S$ and essential in
  its complement.  By hypothesis, $S_2$ has non-empty boundary.  If
  $S_2$ is not connected, replace it by any connected component with
  non-empty boundary.

  Now $M$ has a cell structure $\Tbar$ coming from $\T$ consisting of
  truncated tetrahedra, and note that $S$ is normal with respect to
  $\Tbar$.  Our goal is to normalize $S_2$ in $\Tbar$ and then spin
  the result into a spun-normal surface.  However, not every normal
  surface in $(M, \Tbar)$ can be spun.  The boundary curves need an
  orientation which satisfies the condition in Section~1.12 of
  \cite{Till1}, and that orientation must be compatible with the
  ``tilt'' of the normal discs (see Figure~\ref{fig-unspinable}).
  \begin{figure}
    \begin{center}
      \vspace{0.4cm}
    \includegraphics[scale=1.05]{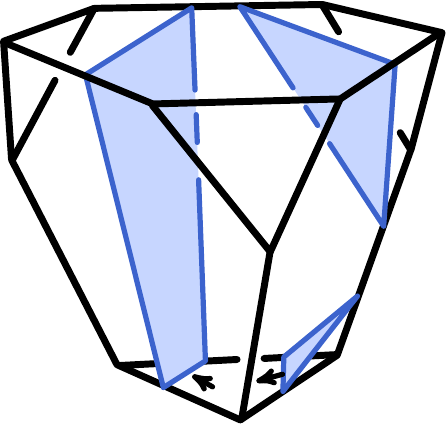} \hspace{0.5cm}
    \includegraphics[scale=0.85]{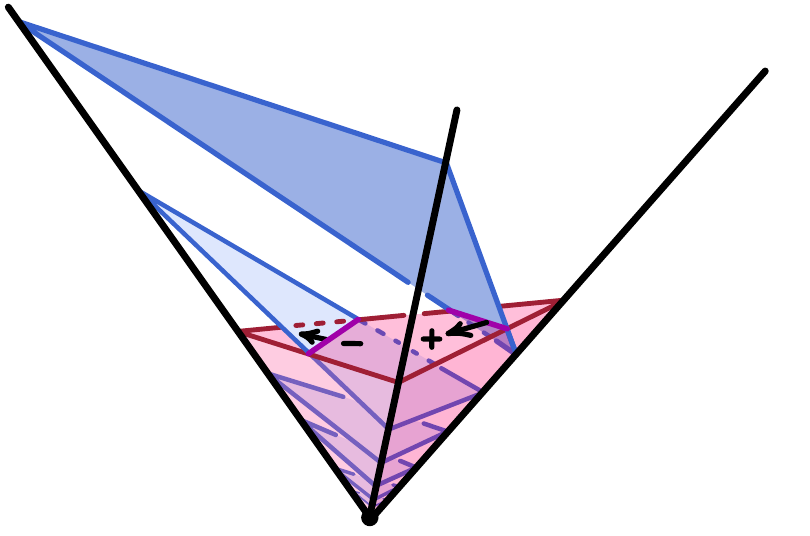}
    \end{center}
  \caption{Some normal discs in a truncated tetrahedron, with a choice of
      orientation for the arcs meeting $\partial M$.
      Comparing with the copy of Figure~\ref{fig-bdry-orient} at right,
      we see the quad can be spun but not the lower triangle.  }
    \label{fig-unspinable}
  \end{figure}
  To finesse this issue, we isotope $S_2$ in the complement of $S$ so
  that $\partial S_2$ consists of normal curves, each of which lies just to
  the positive side of a parallel curve in $\partial S$.
 
  Now normalize $S_2$ with respect to $\Tbar$ to yield a surface $S_3$
  (see e.g.~\cite[Ch.~3]{Matveev2007}).  As mentioned above, this
  normalization takes place in the complement of $S$.  A concise way
  of seeing this is to cut $M$ open along $S$ to yield $M'$ with a
  cell structure $\Tbarprime$.  If we normalize $S_2$ in $M'$ with
  respect to $\Tbarprime$, the result is necessarily normal with
  respect to $\Tbar$.  Moreover the final surface is still disjoint
  from the two copies of $S$ in $\partial M'$ since normalizing never
  increases the number of intersections of the surface with an edge.

  The normalization process may result in compressions or other
  reductions to the surface.  However, since $S_2$ is essential in
  $M'$, it follows that $S_3$ has the same topology as $S_2$.  (If
  $M'$ is irreducible, then $S_2$ and $S_3$ are of course isotopic.)
  Focus on a component of $\partial M \cap M'$, which is an annulus
  $A$.  The components of $\partial S_2$ in $A$ are all normally
  isotopic, and moreover intersect any 2-dimensional face of
  $\Tbarprime$ at most once (see the right half of
  Figure~\ref{fig-unspinable}).  Thus the first $\partial$-compression
  that occurs while normalizing $S_2$ must join two distinct boundary
  components, reducing the total number of boundary components.  As
  $S_2$ and $S_3$ have the same number of boundary components, there
  can be no $\partial$-compressions during normalization and so
  $\partial S_2$ and $\partial S_3$ are setwise the same.
  \begin{figure}[bt]
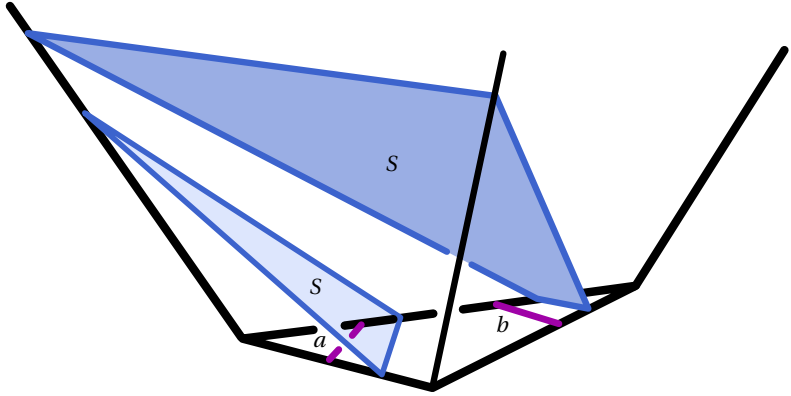

    \begin{cxyoverpic}{(185,93)}{scale=1.6}{images/new-surface}
    ,(117,16)*{b}
    ,(74,12)*{a}
    ,(91,54)*{S}
    ,(73,25)*{S}
\end{cxyoverpic}
\caption{ There are at most two kinds of normal arcs in $\partial
  S_3$, labeled here $a$ and $b$.  From the their position relative to
  the surface $S$, any normal disc of $S_3$ adjacent to $\partial S_3$
  must be parallel to those in $S$.}
    \label{fig-S-vs-S3}
  \end{figure}

  Focus now on one normal arc $\alpha$ in $\partial S_3$.  By construction,
  it lies just to the positive side of a normal arc $\alpha'$ of $\partial
  S$.  If $\n$ is a normal disc of $S_3$ with $\alpha$ as an edge, from
  Figure~\ref{fig-S-vs-S3} we see that $\n$ must be parallel to the normal
  disc of $S$ along $\alpha'$.  Hence, we can build a spun-normal surface
  $S'$ from $S_3$ which is disjoint from $S$ by attaching half-open
  annuli in $\Nv$ which are combinatorially parallel to those of $\Nv
  \cap S$.

  Now $S$ and $S'$ are disjoint spun-normal surfaces in $\T$ and hence
  they have compatible quad types.  Thus $S'$ is in $L(\T, Q)$.  We
  know that both $S$ and $S'$ are nonempty, two-sided, connected, and
  not vertex linking tori.  Hence as $L(\T, Q)$ is one dimensional,
  the $Q$-coordinates of $S$ and $S'$ must be the same.  Hence they
  are normally isotopic and so they have the same topology.  This
  contradicts that we started with a genuine compression of $S$,
  ruling out (a).  Hence $S$ is essential.
\end{proof}

\section{Slopes of alternating knots}
\label{sec-alt}

In this section, we prove Theorem~\ref{thm-alt-nonint} by showing that
the alternating knot $10_{79} = 10a78$ has nonintegral boundary
slopes, namely $10/3$ and $-10/3$.  Additional nonintegral slopes of
alternating knots are given in Table~\ref{table-alt}.  Let $M$ denote
the complement of $10_{79}$; as $M$ is amphichiral, we simply show
that $10/3$ is a boundary slope.
\begin{table}{
    \vspace{1cm}

    \footnotesize
    \begin{center}
      \begin{tabular}{rlrlrl}
10a8:&$-20/3$ & 11a275:&$-20/3$ & 12a120:&$-52/3$\\
10a78:&$-10/3$, $10/3$ & 11a281:&$-28/3$ & 12a125:&$-10/3$, $-2/3$, $2/3$\\
10a95:&$4/3$ & 11a284:&$-2/3$ & 12a126:&$-2/3$\\
11a17:&$-2/3$ & 11a296:&$34/3$ & 12a127:&$-22/3$\\
11a19:&$-2/3$ & 11a299:&$-4/3$ & 12a132:&$-2/3$\\
11a25:&$-2/3$ & 11a300:&$-28/3$, $34/3$ & 12a134:&$2/3$\\
11a38:&$10/3$ & 11a301:&$-22/3$, $34/3$ & 12a154:&$-20/3$\\
11a49:&$-28/3$ & 11a313:&$-20/3$ & 12a155:&$4/3$\\
11a102:&$-16/3$ & 11a314:&$-2/3$ & 12a162:&$-20/3$\\
11a113:&$2/3$ & 11a320:&$-46/3$, $-22/3$ & 12a177:&$10/3$, $22/3$\\
11a125:&$-34/3$ & 11a321:&$20/3$ & 12a186:&$34/3$\\
11a127:&$-40/3$ & 11a323:&$26/3$ & 12a188:&$2/3$\\
11a129:&$40/3$ & 11a326:&$22/3$, $28/3$ & 12a211:&$2/3$\\
11a130:&$-34/3$ & 11a329:&$-4/3$ & 12a222:&$-23/2$\\
11a136:&$-34/3$ & 11a345:&$-10/3$ & 12a223:&$-10/3$\\
11a147:&$-34/3$ & 11a349:&$34/3$ & 12a224:&$-27/2$\\
11a151:&$34/3$ & 12a45:&$26/3$ & 12a233:&$-23/2$\\
11a152:&$-40/3$ & 12a46:&$7/2$, $22/3$ & 12a264:&$-40/3$\\
11a156:&$-4/3$ & 12a52:&$-52/5$ & 12a267:&$-28/3$\\
11a157:&$40/3$ & 12a53:&$-32/3$ & 12a276:&$-20/3$\\
11a158:&$28/3$ & 12a57:&$16/3$ & 12a284:&$4/3$\\
11a162:&$34/3$ & 12a59:&$-8/3$ & 12a292:&$22/3$\\
11a164:&$34/3$ & 12a63:&$-8/3$ & 12a293:&$-14/3$\\
11a168:&$-10/3$ & 12a65:&$26/3$ & 12a294:&$34/3$\\
11a169:&$22/3$ & 12a70:&$34/3$, $46/3$ & 12a296:&$-8/3$\\
11a171:&$-34/3$ & 12a72:&$34/3$, $46/3$ & 12a301:&$-22/3$\\
11a217:&$40/3$ & 12a88:&$-52/3$, $-8/3$, $68/5$ & 12a309:&$22/3$\\
11a218:&$2/3$ & 12a89:&$34/3$ & 12a311:&$-76/3$\\
11a227:&$-2/3$ & 12a91:&$-58/3$, $-8/3$ & 12a315:&$-4/5$, $4/3$\\
11a233:&$28/3$ & 12a93:&$-48/5$ & 12a316:&$36/5$, $28/3$\\
11a239:&$22/3$ & 12a94:&$-32/3$ & 12a317:&$10/3$\\
11a244:&$-2/3$ & 12a100:&$7/2$ & 12a318:&$46/5$, $34/3$\\
11a249:&$-20/3$ & 12a101:&$78/5$ & 12a319:&$-13/2$\\
11a251:&$-4/3$ & 12a102:&$-32/3$, $28/5$ & 12a320:&$-52/3$, $-4/3$\\
11a253:&$-4/3$ & 12a105:&$-8/3$ & 12a321:&$-40/3$, $-26/3$, $-8/3$\\
11a255:&$34/3$ & 12a107:&$-32/3$, $28/5$ & 12a334:&$16/3$\\
11a256:&$-40/3$, $-20/3$ & 12a108:&$-52/3$ & 12a337:&$22/3$\\
11a272:&$-10/3$ & 12a109:&$7/2$ & 12a339:&$-40/3$, $-16/3$\\
11a273:&$22/3$ & 12a111:&$-58/3$, $-8/3$ & 12a340:&$-64/3$\\
11a274:&$-28/3$, $34/3$ & 12a115:&$78/5$ & 12a344:&$-52/3$, $-27/2$\\
\end{tabular}

    \end{center}
}

\vspace{1cm}
\caption{Some nonintegral boundary slopes of alternating knots,
  numbered as in \cite{Knotscape}; the first three are $10_{80},
  10_{79}$, and $10_{106}$ in the standard table \cite{Rolfsen1990}.
  These were proven to exist by Theorem~\ref{thm-one-cusp} using
  triangulations with 14--23 tetrahedra.   }\label{table-alt}
\end{table}

To apply Theorem~\ref{thm-one-cusp}, we need to specify an ideal
triangulation $\T$ with a spun-normal surface $S$ and check:
\begin{enumerate}
\item \label{it-same-mfld}
  The ideal triangulation $\T$ is homeomorphic to the complement of
  $10_{79}$ and the peripheral basis that comes with $\T$ is the
  standard homological one.

\item  \label{it-full-vertex}
  The surface $S$ is a vertex surface with a quad in every
  tetrahedron.  In the reformulation of
  Theorem~\ref{thm-one-cusp-tech}, the former is equivalent to $\dim
  L(\T, Q) = 1$, where $Q$ is the quad type determined by $S$.
  
\item \label{it-bs} The boundary slope of $S$ is $10/3$, which can be
  done as described in Section~\ref{sub-ideal-ess}.
\end{enumerate}
The triangulation $\T$ we use has $14$ tetrahedra and is given in the
file ``10\_79-certificate.tri'' available at \cite{CertTris}.  The surface $S$ has the
same quad type in each tetrahedron, namely the one disjoint from the edges
$01$ and $23$ in Figure~\ref{fig-edge-param}, which also corresponds
to the shape degeneration $z \to 0$.  The number of quads is given by
\[
S = (2, 3, 3, 3, 2, 5, 2, 1, 4, 1, 3, 1, 3, 3) \in \R^{14}_Q
\]
Now (\ref{it-same-mfld}) above is easily checked using SnapPy \cite{SnapPy}.
The information needed for (\ref{it-full-vertex}-\ref{it-bs}) comes
directly from the $A$ part of the matrix $M(\Lambda)$ describing the gluing
equations for $\T$ together with the corresponding part of the cusp
equations.  Explicitly, using SnapPy within Sage \cite{Sage} as 
shown in Table~\ref{tab-code} suffices to confirm Theorem \ref{thm-alt-nonint}.  
\begin{table}
\begin{center}
\begin{minipage}[b]{0.8 \textwidth}
{ \small
\begin{verbatim}
sage: from snappy import *
sage: M = Manifold("10_79-certificate.tri")
sage: N = Manifold("10_79")
sage: M.is_isometric_to(N, return_isometries=True)[1]
0 -> 0
[1 0] 
[0 1] 
Extends to link
sage: data = M.gluing_equations(form="rect")
sage: gluing_data, cusp_data = data[:-2], data[-2:]
sage: A = matrix( [e[0] for e in gluing_data] )
sage: B = matrix( [e[1] for e in gluing_data] )
sage: c = matrix( [ [e[2]] for e in gluing_data] )
sage: cusp_holonomy_A_part = matrix( [e[0] for e in cusp_data] )
sage: L = A.right_kernel(); L
Free module of degree 14 and rank 1 over Integer Ring
Echelon basis matrix:
[2 3 3 3 2 5 2 1 4 1 3 1 3 3]
sage: S = L.basis()[0]
sage: cusp_holonomy_A_part * S
(-3, 10)
\end{verbatim} }
\end{minipage}
\end{center}

\vspace{1ex}

\caption{Checking Theorem \ref{thm-alt-nonint} using SnapPy
  \cite{SnapPy} within Sage \cite{Sage}.}
\label{tab-code}
\end{table}

\section{Dehn filling}\label{sec-filling}

We turn to the case of a \3-manifold $W$ where $\partial W$ consists
of several tori $T_0, T_1, \ldots, T_b$.  For $k > 0$, we pick a slope
$\gamma_k$ on $T_k$.  If we fix an ideal triangulation $\T$ of $W$, we
can consider all spun-normal surfaces $S$ whose boundary slope on
$T_k$ is either $\gamma_k$ or $\emptyset$.  Equivalently, we consider
surfaces $S$ where the geometric intersection of $\gamma_k$ with $S
\cap T_k$ is 0.  By the discussion in Section~\ref{sub-ideal-ess}, for
each $k$ this requirement imposes an additional linear condition on
the cone $C(\T)$ of spun-normal surfaces.  We call the resulting subcone
$C(\T, \{\gamma_k\})$ the \emph{relative normal surface space
  corresponding to $(\, \cdot \, , \gamma_1,\ldots,\gamma_b)$}.  This
section is devoted to:
\begin{maintwobackwards}
  Let $W$ be a compact oriented \3-manifold whose boundary consists of
  tori $T_0, T_1, \ldots , T_b$.  Let $S$ be a spun-normal surface in an
  ideal triangulation $\T$ of $W$, with nonempty boundary slope $\gamma_k$
  on each $T_k$. Suppose that $S$ has a quadrilateral in every
  tetrahedra of $\T$, and is a vertex surface of $C(\T, \{\gamma_1,\ldots,\gamma_b\})$.
  Then $\gamma_0$ is a boundary slope of $W(\cdot, \gamma_1,\ldots, \gamma_b)$.
\end{maintwobackwards}

\begin{proof}
  We consider the relative gluing equation variety $D(\T,
  \{\gamma_k\})$ obtained from adding the $b$ conditions that the
  holonomy $h(\gamma_k)$ of each $\gamma_k$ is 1.  For the Dehn filled
  manifold $M = W( \, \cdot \, ,\gamma_1,\ldots,\gamma_b)$, the
  relative variety $D(\T, \{\gamma_k\})$ is closely related to the
  character variety $\Xbar(M)$.  However, while every point in $D(\T,
  \{\gamma_k\})$ gives a representation $\rho \maps \pi_1(W) \to
  \PSL{2}{\C}$, these representations do not all factor through
  $\pi_1(M)$; the condition $h(\gamma_k) = 1$ only gives that
  $\rho(\gamma_k)$ is trivial or parabolic.  However, $\rho(\gamma_k)$
  can only be nontrivial if $h(\alpha) = 1$ for every element $\alpha
  \in \pi_1(T_k)$.

  As in Section~\ref{rem-deg-spun}, we take the preferred shape
  parameter $z$ in a tetrahedron to be the one where the quad of $S$
  has shift $+1$.  Then following Section~\ref{ex-DT}, we consider the
  variety $V = V(\Lambda)$ arising from $D(\T, \{\gamma_k\})$ by
  focusing on the preferred shape parameters.  If $\T$ has $n$
  tetrahedra, then the rank of $\Lambda$ is at most $n-1$ since there
  are $n - b - 1$ equations coming from $D(\T)$ (by
  Section~\ref{ex-DT}) and also one equation for each condition
  $h(\gamma_k) = 1$ (by Section~\ref{sub-ideal-ess}).  Thus $V$ is
  indeed a variety of the kind studied in
  Section~\ref{sec-ideal-gluing}.  Just as in
  Section~\ref{rem-deg-spun}, the degeneration vectors for $V$ are
  precisely the spun-normal surfaces in the relative space $C(\T,
  \{\gamma_k\})$.  Thus we can apply Proposition~\ref{prop-Kabaya-key}
  to see that the surface $S$ is associated to an ideal point $p$ of
  $V$.  Let $f \maps (D, 0) \to (\Vbar, p)$ be an associated
  holomorphic map.  For each $k \geq 0$, pick a curve $\alpha_k$ on
  $T_k$ which meets $\gamma_k$ in one point.  Then as $\gamma_k$ is
  the boundary slope of $S$, the function $h(\alpha_k) \circ f$ has a
  nontrivial pole or zero at $0$.  In particular, we can restrict the
  domain $D$ of $f$ so that $h(\alpha_k) \neq 1$ on $f(D \setminus \{
  0 \})$.  Then every point in $f(D \setminus \{ 0 \})$ gives rise to
  a representation of $\pi_1(M)$.  Thus we have found an ideal point
  $\xi$ of $\Xbar(M)$ where $\tr \left(\rho(\alpha_0)\right)$ has a
  pole and $\tr\left(\rho(\gamma_0)\right) = \pm 2$.  The essential
  surface associated to $\xi$ has boundary slope $\gamma_0$, as
  needed.
\end{proof}

\section{The 2-fusion link}
\label{sec.thmLm1m2}

Let $W$ be the complement of the link in Figure
\ref{fig-2fusion-link}.  The manifold $W$ has a hyperbolic structure
obtained by gluing two regular ideal octahedra.  We consider a certain
ideal triangulation $\T$ of $W$ with 8 tetrahedra described in the file
``2fusion-certificate.tri'' available at \cite{CertTris}.  As in
Section~\ref{sec-alt}, we look at surfaces with the same quad type in
each tetrahedron, the one which corresponds to the shape degeneration $z \to
0$, and use $Q$ to denote this choice of quads.

One finds that the first  part of the matrix $M(\Lambda)=(A|B|c)$ is 
\[
\scriptsize
A = \left(\begin{array}{rrrrrrrr}
1 & 0 & -1 & 0 & 0 & 0 & 0 & 0 \\
0 & -1 & 0 & 1 & 0 & 0 & 0 & 1 \\
-1 & 1 & 1 & -1 & 0 & -2 & 0 & 0 \\
0 & 0 & 0 & 0 & 0 & 0 & 0 & 0 \\
0 & 0 & -1 & 0 & 0 & 2 & 1 & -1 \\
0 & 0 & 0 & 1 & -1 & 0 & 0 & 0 \\
0 & 0 & 1 & -1 & 1 & 0 & -1 & 0 \\
0 & 0 & 0 & 0 & 0 & 0 & 0 & 0
\end{array}\right)
\]
which has rank 5.
%which row reduces to
%\[
%\scriptsize
%\left(\begin{array}{rrrrrrrr}
%1 & 0 & 0 & 0 & 0 & 0 & -1 & 0 \\
%0 & 1 & 0 & 0 & -1 & 0 & 0 & -1 \\
%0 & 0 & 1 & 0 & 0 & 0 & -1 & 0 \\
%0 & 0 & 0 & 1 & -1 & 0 & 0 & 0 \\
%0 & 0 & 0 & 0 & 0 & 2 & 0 & -1 \\
%\end{array}\right)
%\]
Three vectors which span $\ker A = L(\T, Q)$ are
\begin{align*}
  S_1 &= \big(0,1,0,1,1,0,0,0\big) \\
  S_2 &= \big(0, 2, 0, 0, 0, 1, 0, 2\big)\\
  S_3 &= \big(1, 0, 1, 0, 0, 0, 1, 0\big)
\end{align*}
Thus on $L(\T, Q) = \{ a_1 S_1 + a_2 S_2 + a_3 S_3\}$, the condition
defining $C(\T,Q)$ that the original variables satisfy $u_k \geq 0$
translates into having each $a_k \geq 0$.   Hence $C(\T, Q)$ is simply the positive
orthant in $L(\T, Q)$ with respect to the basis $\{S_1, S_2, S_3\}$.  So we
can identify the projective solution space $P(\T, Q)$ with the
triangle spanned by the vertex surfaces $S_k$.  

Now with the peripheral basis curves ordered $(\mu_0, \lambda_0, \mu_1, \lambda_1,
\mu_2, \lambda_2)$, the $A$ part of the matrix specifying the cusp equations
is
\[
\scriptsize
\left(\begin{array}{rrrrrrrr}
0 & -1 & 0 & 0 & 0 & 0 & 0 & 0 \\
-1 & 2 & 0 & 0 & 0 & 0 & 0 & 0 \\
0 & 0 & 0 & 0 & 0 & 0 & -1 & -1 \\
0 & 0 & 0 & 0 & 0 & 0 & 0 & -1 \\
0 & 0 & 1 & 0 & 0 & 0 & 0 & 0 \\
0 & 1 & 0 & 0 & 0 & 0 & 0 & -1
\end{array}\right)
\]
and hence the boundary slopes of each of our vertex surfaces is 
\begin{equation}\label{eq-W-bs}
\begin{array}{rcccc}
 & & T_0  & T_1  & T_2 \\ \hline
\partial S_1 &: &2 \mu_0 + \lambda_0  & \emptyset & \mu_2 \\
\partial S_2 &: &4 \mu_0 + 2\lambda_0& -2 \mu_1 + 2 \lambda_1& \emptyset \\
\partial S_3 &: &-\mu_0& \lambda_1& -\lambda_2
\end{array}
\end{equation}
We will show
\begin{proposition}\label{prop-W-bs}
  The surface $S = a_1 S_1 + a_2 S_2 + a_3 S_3$ for $a_k \in \N$ has non-empty boundary
  slopes $\gamma_0, \gamma_1, \gamma_2$ on each boundary torus $T_k$, and $\gamma_0$
  is a boundary slope of $M = W(\, \cdot \, , \gamma_1, \gamma_2)$.
\end{proposition}
\begin{proof}
  Since all $a_k> 0$, it is clear from (\ref{eq-W-bs}) that $\partial S$ has
  nontrivial coefficients along each $\lambda_k$ and so has non-empty
  boundary slope $\gamma_k$ on each $T_k$.  Consider the boundary slope
  map from the convex hull of the $S_k$ to the space $\R^4 = H_1(T_1;
  \R) \oplus H_1(T_2; \R)$.  From (\ref{eq-W-bs}), it is clear this is
  injective, even if we projectivize the image.  Thus, the relative
  normal surface space $C(\T, \{\gamma_1, \gamma_2\})$ is just the ray
  generated by $S$, and so $S$ is a vertex surface for $C(\T,
  \{\gamma_1, \gamma_2\})$.  Hence we can apply Theorem~\ref{thm-filling}
  to see that $\gamma_0$ is a boundary slope of $M$.
\end{proof}

We now prove Theorem~\ref{thm.Lm1m2} by considering the surface
\[
S = 2(m_1-1)  S_1  + m_2  S_2 + 2(m_1-1)m_2  S_3
\]
for some $m_1 > 1$ and $m_2 > 0$.  This surface has boundary slopes as
follows, written as elements of $\Q$:
\[
\gamma_0 = -\frac{{{\left(m_{1} - 3\right)} m_{2} - 2 \, m_{1} + 2}}{{m_{1} + m_{2} - 1}}, \quad \gamma_1 = -\frac{1}{m_{1}}, \quad \gamma_2 = -\frac{1}{m_{2}}
\]
Thus by Proposition~\ref{prop-W-bs}, the slope $\gamma_0$ above
is a boundary slope for
\[
M = W\big(\, \cdot \, , -1/m_1, -1/m_2 \big)
\]
Now, the manifold $M$ is the exterior of the knot $L(m_1, m_2)$ from
Theorem~\ref{thm.Lm1m2}, but the peripheral basis $\{\mu_0, \lambda_0 \}$ is
the one that comes from $W$, and so $\lambda_0$ is \emph{not} the homological
longitude $\lambda'_0$for $M$.  As the components $C_1$ and $C_2$ are
unlinked, we can adjust for this via
\[
\lambda_0' = \left(-m_1 \cdot \mathrm{lk}(C_0, C_1)^2 -m_2 \cdot \mathrm{lk}(C_0, C_2)^2\right) \mu_0 + \lambda_0 = (-4 m_1 - 9 m_2) \mu_0 + \lambda_0
\]
and thus find that in the usual homological basis
\[
\gamma'_0 = (4 m_1 + 9 m_2)  + \gamma_0 = 3 (m_1 + 1) + 9 m_2 + \frac{(m_1-1)^2}{(m_1 + m_2-1)}
\]
is a boundary slope of $L(m_1, m_2)$, proving Theorem~\ref{thm.Lm1m2}.   

\section{Which spun-normal surfaces come from ideal points?}
\label{sec-which-come-from-ideal}

Given an ideal triangulation $\T$ of a \3-manifold $M$ with one torus
boundary component, we would like to determine all the boundary slopes
that arise from ideal points of $D(\T)$.  Of course, one can find all
such detected slopes by computing the $A$-polynomial, but this is
often a very difficult computation, involving projecting an algebraic
variety (i.e.~eliminating variables).  While Culler has a clever new
numerical method for such computations \cite{Culler}, there are still
9 crossing knots whose $A$-polynomials have not been computed.

For a spun-normal surface $S$ with nonempty boundary, we have an
effectively checkable condition (Lemma~\ref{lem-gen-char}) which is
necessary and sufficient for $S$ be associated to an ideal point of
$D(\T)$.  However, since there are typically infinitely many
spun-normal surfaces, this does not allow for the computation of all
such detected slopes unless we can restrict to a finite set of
surfaces.  From the point of view of normal surface theory, two
natural finite subsets are:
\begin{enumerate}
\item The vertex surfaces introduced in Section~\ref{sec-spun-normal}.
\item The larger set of fundamental surfaces, which are the integer
  points in $C(\T)$ which are not proper sums of other such points.
  %The
  %$\Z_+$ span of these includes all the integer points of $C(\T)$, see
  %e.g.~\cite[Thm~3.2.8]{Matveev2007}.
\end{enumerate}
However, we show below that neither of these subsets suffices.  In
fact, there is a geometric triangulation $\T$ of the complement of the
knot $6_3$ where \emph{none} of the 22 fundamental surfaces is
associated with an ideal point of $D(\T)$!

Independent of this issue, we've also seen three conditions which
ensure that a surface $S$ is associated to an ideal point of $D(\T)$:
\begin{enumerate}
\item \label{it-kab}
  Kabaya's original criterion, Proposition~\ref{prop-Kabaya-key},
  which was used in proving Theorem~\ref{thm-one-cusp}.  This requires
  that $S$ is a vertex surface and has a quad in every tetrahedron.

\item \label{it-W0}
  Lemma~\ref{thm.genuine2} applies when the first-order system
  $W_0(\Lambda, d)$ has dimension 0, where $d$ is the degeneration vector
  corresponding to $S$.

\item \label{it-W0bar} 
  Lemma~\ref{lem-gen-char} applies when $\Wbar_0(\Lambda, d)$ is nonempty.
\end{enumerate}
Here, each condition implies the next, and (\ref{it-W0bar}) is
necessary as well as sufficient.  Condition (\ref{it-W0}) is easier to
check than (\ref{it-W0bar}) as it only needs the dimension of a
variety, which is one of the easiest tasks for Gr\"obner bases.  In
contrast, (\ref{it-W0bar}) requires eliminating a variable, albeit one
that appears only in fairly simple equations, and thus
(\ref{it-W0bar}) is still much easier than finding the $A$-polynomial
using Gr\"obner bases.  For manifolds with less than 20 tetrahedra,
tests (a) and (b) are usually quite feasible for any given surface.
However, a naive Gr\"obner basis approach to applying (c) sometimes
failed even for manifolds with less than 10 tetrahedra.

\subsection{Experimental Data}

There are 173 Montesinos knots with $< 11$ crossings.  As we know the
boundary slopes for these \cite{HO, DunfieldMontesinos}, we tested the
three methods above on each of them, using triangulations with between
2 and 15 tetrahedra.  These knots have an average of 6.1 boundary
slopes, but the method (\ref{it-kab}) yields an average of only 1.2
slopes, or about 20\% of the total.  When (\ref{it-W0}) is applied to
all vertex surfaces, it finds an average of 3.8 slopes, or about 64\%
of the actual number of slopes.  The third test (\ref{it-W0bar}) was
not practical on enough of these vertex surfaces to give any real
data.

When the manifolds were ordered by the size of their triangulations,
the number of slopes found by (\ref{it-kab}) decreased (in absolute
and relative terms) as the number of tetrahedra increased.  A more
marked variant of this pattern was observed in punctured torus
bundles; when the triangulations were small there was an average of
1.0 slope found with (\ref{it-kab}), but when there were 15 tetrahedra the
average had dropped to below < 0.1. 

Also for punctured torus bundles, method (\ref{it-W0}) always found
exactly two slopes.  Henry Segerman pointed out to us that these are
the two surfaces corresponding to the edges of the Farey strip in
\cite{FloydHatcher}.  This can be deduced from \cite{Segerman2011}
where the solutions to the tilde equations of Section 8 are closely
related to our $W(\Lambda,d)$.  Another interesting observation of
Segerman is the following simple way that (\ref{it-W0}) that can fail.
If the detected surface $S$ has a tube of quads encircling an edge,
then all the edge parameters around it are 1 at the ideal point.  Thus
the equation (\ref{eq-first-order}) for that edge is simply $1 = 1$
and so the dimension of $W_0(\Lambda, d)$ will be at least one if it
is not empty, and hence (\ref{it-W0}) will not apply.  Of course, such
an $S$ has an obvious compression from the tube of quads (which is
typically a genuine compression), but for the examples in
\cite{Segerman2011} the spun-normal surfaces associated to ideal
points frequently do have such tubes.

\subsection{The knot $6_3$}\label{sec-funny-knot}

We illustrate some of the subtleties of these questions with the
complement $M$ of the hyperbolic knot $6_3$ in $S^3$.  Using that this
is the two-bridge knot $K(5/13)$, one finds that the boundary slopes
are: $-6, -2, 0, 2, 6$.  (The symmetry here comes from the fact that
$M$ is amphichiral.)

From the $A$-polynomial, we see that the character variety $\Xbar(M)$
has a single irreducible component (excluding the component of
reducible representations).  All boundary slopes are strongly detected
on $\Xbar(M)$, with the exception of $0$ which com1es from a fibration
of $M$ over the circle.  We focus on the boundary slope 2, which is
associated to a unique ideal point of $\Xbar(M)$.

From now on, let $\T$ be the canonical triangulation of $M$ as saved
in ``6\_3-canon.tri'' available at \cite{CertTris}.  It has 6
tetrahedra, which come in three different shapes.
\begin{itemize}
  \item Tets 0 and 2 have the same shape, which is an isosceles triangle.
  \item Tets 1 and 3 have the same shape, which has no symmetries.
  \item Tets 4 and 5 have the same shape, which is the mirror image of
    those of tets 1 and 3.
\end{itemize}
All of this is compatible with the fact that the isometry group of $M$
is the dihedral group with eight elements.  It turns out that there are 16
spun-normal vertex surfaces, all of which have non-trivial boundary
slopes, and also 4 other fundamental surfaces.

Four vertex surfaces have slope $2$, all of which are
compatible with a single quad-type
\[
	Q = [Q03, Q13, Q13, Q03, Q13, Q13]
\]
and have weights
\begin{align*}
 S_8  &= [0, 1, 2, 0, 1, 1]  &  S_{10} &= [2, 1, 0, 0, 1, 1] \\
 S_9  &= [0, 0, 2, 1, 1, 1]  &  S_{11} &= [2, 0, 0, 1, 1, 1]
\end{align*}
While each of these vertex surfaces has exactly one boundary
component, they differ in the direction the surface is spun out to the
boundary.  The surfaces $S_8$ and $S_9$ are spun one way, and $S_{10}$
and $S_{11}$ are spun the other.  Additionally, there are two other
fundamental surfaces in this component of $F(\T)$:
\begin{align*}
  S_a  &= (1/2) \big( S_8 + S_{10} \big) = [1, 1, 1, 0, 1, 1]  \\
  S_b  &= (1/2) \big( S_9 + S_{11} \big) = [1, 0, 1, 1, 1, 1]
\end{align*}

\subsection{Oddity the first} The surfaces $S_8$ and $S_{10}$ are
compatible and each has nonempty boundary, but $S_8 + S_{10}$ is
actually a closed surface.  In fact, it's the double of $S_a$ which
has genus 2 and is just the boundary torus plus a tube linking the
edge $e_3$.  This is in stark contrast with the non-spun case, where
compatible normal surfaces with boundary always sum to a surface with
nonempty boundary.  (This is because the two surfaces lie on a common
branched surface.)  Thus this is a potential problem for proving that
all boundary slopes can be determined solely by looking at the
spun-normal vertex surfaces.

\subsection{Oddity the second} None of the 22 fundamental surfaces arise
from an ideal point of the gluing equation variety $D(\T)$.  For
instance, for the surfaces with slope 2, chose the preferred edge
parameters so that $z_i \to 0$ corresponds to the quad in $Q$.  Then
the gluing equations include $z_1 = z_3$ and $z_4 = z_5$; the former
is not compatible with any of the above fundamental surfaces.
Instead, after some work it turns out that Lemma~\ref{lem-gen-char}
shows that the surfaces
\begin{equation}\label{eq-surfaces}
\begin{split}
	S =  S_{8\hphantom{0}}+ S_b &=  S_{9\hphantom{1}} + S_a = [1, 1, 3, 1, 2, 2]  \\ 
        S' = S_{10} + S_b &= S_{11} + S_a = [3, 1, 1, 1, 2, 2] 
\end{split}
\end{equation}
are associated to the two ideal points of $D(\T)$ which detect the
slope +2.  (These two ideal points map to the same ideal point of
$\Xbar(M)$, and differ in the direction the associated surfaces are spun
out toward the boundary.)

A posteriori, the failure of the fundamental surfaces to appear at
ideal points of $D(\T)$ is not so surprising given the large symmetry
group $G$ of $\T$.  The four vertex surfaces above are the vertices of
a tetrahedron $\Delta$ in the projectivized space $\mathit{PF}(\T)$ of
embedded spun-normal surfaces. The subgroup $H$ of $G$ which preserves
$\Delta$ is isomorphic to $\Z/2 \oplus \Z/2$ and acts transitively on
the vertices of $\Delta$ by orientation preserving symmetries.  Now
$D(\T)$ has two ideal points with slope 2 and there is a unique
non-trivial element $g$ of $H$ which fixes both; this $g$ acts by $S_8
\leftrightarrow S_9$ and $S_{10} \leftrightarrow S_{11}$
(i.e.~interchanges the pairs of surfaces that spin in the same
direction).  Thus a surface associated with either ideal point must
lie on the line segment joining $(1/2)(S_8 + S_9)$ to $(1/2)(S_{10} +
S_{11})$ and hence can not be a fundamental surface.  However, since we
also know that $g$ interchanges $S_a$ and $S_b$, we can see that the
surfaces $S$ and $S'$ in (\ref{eq-surfaces}) will indeed be fixed by $g$.  

{\RaggedRight \bibliographystyle{math-1.6} \bibliography{2fusion.bib} }

\end{document}